%% file: GlauPreprint_Sob-index2012_0304.tex
\documentclass[11pt,a4paper,reqno,intlimits,sumlimits]{amsart}

\usepackage{amsmath}
\usepackage{amssymb}
\usepackage{chicago}
\usepackage[mathscr]{euscript}
\usepackage{enumerate}
\usepackage{xspace}
\usepackage{color}
\usepackage{epsfig,rotating, url}

\usepackage{a4wide, amsmath, amsxtra,  dsfont, exscale, fancybox, latexsym, verbatim }
\usepackage[utf8]{inputenc}
\DeclareMathAlphabet{\mathpzc}{OT1}{pzc}{m}{it}

\input{kopf-defs.tex}

\newcommand{\zitep}[1]{\relax}
\newcommand{\notiz}[1]{{\protect\makebox[0cm][l]{$\!^\Join$}}\marginpar{\tiny {$\Join$} #1}}
\renewcommand{\notiz}[1]{\relax}

\theoremstyle{plain}
\newtheorem{theorem}{Theorem}[section]
\newtheorem{lemma}[theorem]{Lemma}
\newtheorem{proposition}[theorem]{Proposition}

\newtheorem{defin}[theorem]{Definition}

 \newtheorem{remark}[theorem]{Remark}
 \newtheorem{example}[theorem]{Example}

\begin{document}


\title[Sobolev index: A classification of L\'evy processes]{Sobolev index: A classification of L\'evy processes \\via their symbols}
\author[K. Glau]{Kathrin Glau}

\address{Technische Universit\"at M\"unchen, Parkring 11
85748 Garching b. M\"unchen
}
\email{kathrin.glau@tum.de}



\thanks{Parts of the results presented in this article were derived in the Ph.D. thesis \cite{PhdGlau}. The author expresses her gratitude to Ernst Eberlein for his valuable support as advisor of the Ph.D. thesis and moreover for his various useful advices as her academic teacher.
The author further thanks Mathias Beiglboeck, Carsten Eilks and Antonis Papapantoleon for critical reading and for fruitful comments on the manuscript. Financial support of the DFG through project \mbox{EB66/11-1} is gratefully acknowledged.}

\date{\today}
\maketitle

\setcounter{tocdepth}{1}

\frenchspacing
\pagestyle{myheadings}

\begin{abstract}

We classify L\'evy processes according to the solution spaces of the associated parabolic PIDEs. This classification reveals structural characteristics of the processes and is relevant for applications such as for solving PIDEs numerically for pricing options in L\'evy models.

The classification is done via the Fourier transform i.e. via the symbol of the process. We define the Sobolev index of a L\'evy process by a certain growth condition on the symbol. It follows that for L\'evy processes with Sobolev index $\alpha$ the corresponding evolution problem has a unique weak solution in the Sobolev-Slobodeckii space $H^{\alpha/2}$.
We show that this classification applies to a wide range of processes. Examples are the Brownian motion with or without drift, generalised hyperbolic (GH), CGMY and (semi) stable L\'evy processes.

A comparison of the Sobolev index with the Blumenthal-Getoor index sheds light on the structural implication of the classification. More precisely, we discuss the Sobolev index as an indicator of the smoothness of the distribution and of the variation of the paths of the process. This highlights the relation between the $p$-variation of the paths and the degree of smoothing effect that stems from the distribution.
\end{abstract}

\input{content}
\bibliographystyle{chicago}
\bibliography{literatur} 
\end{document}

%% file: kopf-defs.tex
\newcommand{\dd}[1]{\operatorname{d}\!#1}
\newcommand{\ee}[1]{\operatorname{e}^{#1}}

\newcommand{\1}{\mathds 1}            

\newcommand{\rr}{\mathds R}
\newcommand{\rrd}{\mathds{R}^d}

\newcommand{\nn}{\mathds N}

\newcommand{\cc}{\mathds C}
\newcommand{\ccd}{\mathds{C}^d}

\newcommand{\OA}{\mathcal{A\,}}

\newcommand{\OG}{\mathcal{G\,}}

\newcommand{\ghut}{\hat{g}}

\newcommand{\uh}{\hat{u}}
\newcommand{\uhut}{\hat{u}}

\newcommand{\vh}{\hat{v}}
\newcommand{\vhut}{\hat{v}}

\newcommand{\muhut}{\hat{\mu}}

\newcommand{\skl}{\langle}
\newcommand{\skr}{\rangle}

\DeclareMathOperator{\loc}{loc}

\newcommand{\Rkri}{\overset{\kern+.5ex\circ}{R}\vphantom{l}}

\newcommand{\QQkri}{\overset{\kern+.5ex\circ}{Q}\vphantom{l}}

\newfont{\schoen }{callig15 at 13.5pt}



%% file: content.tex
\section{Introduction}

\noindent The Feynman-Kac formula provides a fundamental link between conditional expectations and solutions to PDEs.
Under suitable regularity assumptions, a Feynman-Kac representation relates certain conditional expectations to weak solutions of parabolic equations. In financial mathematics this fact is used to compute option prices by solving parabolic equations.

 In the context of L\'evy processes, conditional expectations are linked to solutions of Partial Integro Differential Equations (PIDEs).
In \cite{MatachePetersdorffSchwab2004}, \cite{MatacheSchwabWihler2005}, \cite{MatacheNitscheSchwab2005} wavelet-Galerkin methods for pricing European and American options have been developed.  The methods have been extended to multivariate models, see \cite{ReichSchwabWinter2008}, \cite{phdWinter} and the references therein.
Also standard finite element methods are efficiently used for pricing basket options, even in high dimensional models using dimension reduction techniques, see \cite{hepperger2010} and \cite{hepperger2011c}. \cite{Achdou2008} provides a calibration procedure of a L\'evy model based on PIDEs.
Essential for those finite element methods is the existence and uniqueness of a weak solution of the PIDE, related to the underlying process, in a certain Sobolev-Slobodeckii space.

In other words, a relation between L\'evy processes and Sobolev-Slobodeckii spaces is seminal. More precisely, L\'evy processes of a certain type are linked to Sobolev-Slobodeckii spaces $H^s$ with a certain index $s>0$. This index is important since it classifies the nature of the related evolution problems, resp. of its weak solutions.
It turns out that if the symbol of the L\'evy process satisfies certain polynomial growth conditions with degree $2s$, then the evolution problem has a weak solution in the space $H^s$.  The structural connection between certain types of L\'evy processes and Sobolev-Slobodeckii spaces is thus reflected by the index $s$. This leads us to the definition the \emph{Sobolev index} of the L\'evy process.

It is worth mentioning that in the classical theory on weak solutions of evolution problems, existence and uniqueness of a weak solution are related to the so-called G{\aa}rding and continuity inequalities of the bilinear form. The bilinear form is given via the operator of the equation. While PIDEs are classified via their operators, L\'evy processes are determined by their characteristic functions due to the famous L\'evy-Khintchine formula. Various classes of L\'evy processes, as e.g. the CGMY processes, are directly defined by specifying their characteristic functions. The symbol of a L\'evy process is given via the exponent of the Fourier transform of the process, resp. in terms of the cumulant generating function, see e.g. \cite{Jacob.I}. Therefore properties of the symbol can be canonically derived for a wide range of L\'evy processes.

Crucial for connecting both approaches is Parseval's equality that allows to express the bilinear form associated to the infinitesimal generator via the symbol; details are provided in Section 2, where the notation and this connection is formally shown. In Section 3 the argument is outlined in detail.

For various classes of L\'evy processes we compute the Sobolev index in Section 4. The Brownian motion with and without drift has Sobolev index $2$. We show that the generalised hyperbolic (GH) processes, Cauchy processes, Student-$t$ processes, and the multivariate NIG processes have Sobolev index $1$. The Sobolev index is additionally discussed for CGMY processes, and for L\'evy processes without continuous martingale part which have an absolutely continuous L\'evy measure.

The symbol of a generic $\alpha$-stable L\'evy process is of the form $c|u|^\alpha$ with a positive constant $c$, hence it is polynomial and the Sobolev index can be deduced in an obvious way. In Section 5, we will shed light on the Sobolev index for the wider class of $\alpha$-semi-stable L\'evy processes.

The last section is dedicated to the examination of the Sobolev index in connection to the Blumenthal-Getoor index. For L\'evy processes that have a Sobolev index smaller than $2$, we derive that the Blumenthal-Getoor index is bigger or equal to the Sobolev index. Thereby a link is established between path properties of the process and the smoothing effect of the related evolution problem. Moreover, in view of the Feynman-Kac representation a link to the smoothing effect of the distribution is provided.

\section{The infinitesimal generator and the symbol of a L\'evy process}
\noindent Let $L$ be a L\'evy process with values in $\rrd$ and characteristics $(b,\sigma,F)$ with respect to a truncation function $h$. Here, a measurable function $h:\rr^d\to\rr$ is called a truncation function if $h(x)=x$ in a neighbourhood of $0$.

The distribution  of the process is uniquely determined by the distribution $\mu_t:=P^{L_t}$ for any (for some) $t>0$ and hence by the characteristic function $\muhut_t$ of $L_t$,
\begin{align}\label{eq-charlevy}
\muhut_t(\xi)=E\ee{i \skl \xi, L_t \skr} = \ee{ t \theta(i\xi)}\,.
\end{align}
with cumulant generating function
\begin{align}\label{eq-charexpolevy}
\theta(i\xi) = - \frac{1}{2}\langle \xi, \sigma \xi\rangle + i\langle \xi,b \rangle 
+ \int\left(\ee{i\langle \xi,y\rangle} -1- i\langle \xi,h(y)\rangle\right)\,F (\dd y)\,,
\end{align}
where we denote by $\langle \cdot, \cdot\rangle$ the Euclidean scalar product in $\rr^d$. The matrix $\sigma$ is a symmetric, positive semidefinite $d\times d$-matrix, $b\in \rr^d$ and $F$ is a L\'evy measure i.e. a Borel measure on $\rrd$ with $\int (|x|^2\wedge 1) F(\dd x)<\infty$.
\\
Furthermore we denote by $\OG$ the infinitesimal generator of the process $L$, i.e.
\begin{eqnarray}\label{levygenerator}
\OG f(x)&=&\frac{1}{2}\sum_{j,k=1}^d \sigma^{j,k}\frac{\partial^2 f}{\partial x_j\partial x_k}(x)+\sum_{j=1}^d b^j\frac{\partial f}{\partial x_j}(x)\\
\nonumber
&&\,+\int_{\rr^d}\Big( f(x+y)-f(x)-\sum_{j=1}^d\frac{\partial f}{\partial x_j}(x) \, \big( h(y)\big)_j\Big)F(\dd y)
\end{eqnarray}
for $f\in C^2_0(\rr ^d)$.
We define
$$
\OA:=-\OG\,.
$$
The \emph{symbol $A$ of the process $L$} is defined by
\begin{eqnarray*}
A(\xi) &:=& \frac{1}{2}\langle \xi,\sigma\,\xi\rangle + i\langle \xi,b\rangle 
- \int\left(\ee{-i\langle \xi,y\rangle} -1+ i\langle \xi,h(y)\rangle\right)\,F(\dd y) \\
&=& -\theta(-i\xi) \,,
\end{eqnarray*}
compare e.g. \cite{Jacob.I}.
We have
\begin{equation}\label{mhut=eA}
\muhut_t(\xi) = E \ee{i \langle \xi, L_t \rangle} = \ee{- t A(-\xi) }  \,.
\end{equation}
Let us further denote by $S(\rrd)$ the Schwartz space  i.e. the set of smooth functions $\varphi\in C^\infty(\rr^d,\cc)$
with 
$$
(1+|x|^m) |D^\alpha \varphi(x)| \rightarrow 0\,,|x|\rightarrow \infty
$$
for every multi index $\alpha = (\alpha_1,\ldots,\alpha_d) \in\nn_0^d$ and every $m\in \nn_0$, where $D^\alpha$ denotes the multiple partial derivative
$$
D^\alpha\varphi(x) := \frac{\partial^{\alpha_1}\cdots\partial^{\alpha_d}} {\partial x_1^{\alpha_1} \cdots\partial x_d^{\alpha_d}} \varphi (x)\,.
$$

Let us sketch a relation between the Fourier transform of the distribution, the symbol of the process and a partial integro differential equation.
Let $u\in S(\rr^d)$ and
$$
T_t u(x) := E_x \big( u(L_t)\big) := E\big( u(L_t+x)\big) 
$$
If the absolute value of the characteristic function $\muhut:\rr^d\rightarrow \cc$ is bounded by a polynomial, then Parseval's equality yields
$$
T_t u (x)  = \frac{1}{(2\pi)^d} \int \ee{-i\langle \xi,x\rangle } \muhut_t(-\xi) \uhut(\xi) \dd \xi \,.
$$
In particular, $(T_t)_{t\ge 0}$ is a family of pseudo differential operators $T_t$ with symbol $\muhut_t(-\cdot)$. Changing the order of integration  and differentiation, we obtain
$$
\OG u (x) =  \lim_{t\rightarrow 0}  \frac{T_t u - u}{t}(x) = \partial_t \big(T_tu(x)\big)\big|_{t=0}
=  \frac{1}{(2\pi)^d} \int \ee{-i\langle x,\xi \rangle } A(\xi) \uhut(\xi) \dd \xi \,,
$$
where $\uhut$ denotes the Fourier transform of $u$.
Hence the infinitesimal generator $\OG$, which satisfies
$$
\OG u =  \lim_{t\rightarrow 0}  \frac{T_t u - u}{t}\,,
$$
compare e.g. \citeN{Dynkin1965} and \citeN[Chapter 4]{Jacob.I} is a pseudo differential operator with symbol $-A$ resp.
$$
\OA u (x) = - \OG u(x) =   \frac{1}{(2\pi)^d} \int \ee{-i\langle x,\xi \rangle } A(\xi) \uhut(\xi) \dd \xi \,.
$$

Let us first notice that the symbol $A$ of a L\'evy process is a Borel measurable function $A:\rrd\to \cc$ and
there exists a positive constant $C>0$ such that
\begin{equation}\label{A<=pol2}
|A(\xi)| \le C\big(1+ |\xi|\big)^2\qquad(\text{for all }\xi \in \rrd)\,,
\end{equation}
which is well-known and standard to verify\notiz{\cite[Satz I.7 and Bemerkung I.8 d)]{PhdGlau}}. According to the notation in  \cite{Eskin}, we say that $A\in S^0_2$. More generally, we write $A\in S^0_\alpha$, if $|A(\xi)| \le C\big(1 + |\xi|\big)^\alpha$ for a certain $\alpha\in\rr$ and a constant $C\ge0$.

\notiz{
It is standard to verify the following remark.
\begin{rem}\label{alpha-levy -lel2}
If $A$ is the symbol of a L\'evy process, then there exists a positive constant $C>0$ such that
$$
|A(\xi)| \le C\big(1+ |\xi|\big)^2\qquad\text{for all }\xi \in \rrd\,.
$$
\end{rem}
\begin{proof}
It is obviously sufficient to show the assertion for the integral part of $A$. Furthermore we can choose the truncation function to be $h(y):=y\1_{\{|y|<1\}}$.
With the triangle inequality we decompose the integral
$$
\left| \int\left(\ee{-i\langle \xi,y\rangle} -1+ i\langle \xi,h(y)\rangle\right)\,F(\dd y) \right|
$$
in three parts, an integral over $\rrd\setminus(-1,1)$, a second one over the set $[-1,1]^d\setminus (-\epsilon,\epsilon)^d$ and the last one over $(-\epsilon,\epsilon)^d$. By \citeN[Lemma 8.6]{Sato} there is a function $\theta=\theta(\xi,y)$ in $\cc$ with $|\theta|\le 1$ such that
$$
\ee{-i\skl \xi,y \skr} = 1 -  i \skl \xi,y \skr + \theta(\xi,y) \frac{|\skl \xi,y \skr|^2}{2}\,.
$$
It follows that
\begin{eqnarray*}
 \int_{(-\epsilon,\epsilon)^d} \Big|\ee{-i\skl \xi,y \skr} - 1 + i \skl \xi,y \skr \Big| F(\dd y) 
&=&
\int_{(-\epsilon,\epsilon)^d}\Big|\theta(\xi,y) \frac{|\skl \xi,y \skr|^2}{2}\Big| F(\dd y)
\le
\frac{|\xi|^2}{2}\int_{(-\epsilon,\epsilon)^d} |y|^2 F(\dd y)\,.
\end{eqnarray*}
For the second term we have
\begin{eqnarray*}
 \bigg| \int_{[-1,1]^d\setminus(-\epsilon,\epsilon)^d}\left(\ee{-i\langle \xi,y\rangle} -1 + i\skl \xi,y\skr \right)\,F(\dd y) \bigg| 
\le
2 F\big([-1,1]^d\setminus(-\epsilon,\epsilon)^d\big) + |\xi|\int_{[-1,1]^d\setminus (-\epsilon,\epsilon)^d}|y| F(\dd y)\,,
\end{eqnarray*}
and for the last term we immediately obtain
\begin{eqnarray*}
 \bigg| \int_{\rrd\setminus[-1,1]^d}\left(\ee{-i\langle \xi,y\rangle} -1\right)\,F(\dd y) \bigg|
\le
2 F\big(\rrd\setminus [-1,1]^d\big)\,.
\end{eqnarray*}
To summarize, depending on $\epsilon$ we get constants $C_1'(\epsilon),\,C_2'(\epsilon),\,C_3'(\epsilon) >0$ and hence a positive constant $C>0$ such that
\begin{eqnarray*}
\left| \int\left(\ee{-i\langle \xi,y\rangle} -1+ i\langle \xi,h(y)\rangle\right)\,F(\dd y) \right|
&\le&
C_1'(\epsilon) |\xi|^2 + C_2'(\epsilon)|\xi| + C_3(\epsilon) \\
&\le&
C\big(1+|\xi|\big)^2 \,.
\end{eqnarray*}
\end{proof}
}

Let $g\in S(\rr^d)$ and
$$
v(t,x) := E\big( g(L_T)\,\big|\,L_t=x\big)\,,
$$
then
\label{Arg-PIDE=Ewert}
\begin{equation}\label{gl-faltformel}
v(t,x)
= E\big( g(L_{T-t}\!+\!x)\big)
=
\frac{1}{(2\pi)^d} \int \! \ee{-i\langle \xi,x\rangle} \muhut_{T-t}(-\xi)\ghut(\xi)\dd\xi
\end{equation}
and hence $\vhut(t,\xi) = \ee{(T-t) A(\xi)} \ghut(\xi)$.
On the other hand we have
\begin{eqnarray*}
 \partial_t v(t,x)
=
\partial_t \big(T_{T-t} g(x) \big)
&=&
\frac{1}{(2\pi)^d}  \int \ee{-i\langle x,\xi\rangle} \left( \partial_s\ee{(T-s)A(\xi)}\big|_{s=t}\right) \ghut(\xi) \dd \xi \\
&=&
\frac{1}{(2\pi)^d}  \int \ee{-i\langle x,\xi\rangle} A(\xi) \vhut(t,\xi) \dd\xi \\
&=&
-\OG v(t,x) \,.
\end{eqnarray*}
In other words, the function $v$ satisfies the PIDE
\begin{eqnarray*}
 \partial_t v(t,x) + \OG v (t,x) &=& 0 \qquad\text{ for all }(t,x)\in(0,T)\times\rr^d \, \\
v(T,x) &=& g(x)\quad\text{for all }x\in\rr^d\,.
\end{eqnarray*}
For $V(t,x):=v(T-t,x) = E\big( g(L_T)\,\big|\,L_{T-t}=x\big)$ we accordingly have
\begin{eqnarray}\label{gl-musterparab}
 \partial_t V(t,x) + \OA V (t,x) &=& 0 \qquad\text{ for all }(t,x)\in(0,T)\times\rr^d \, \\
V(0,x) &=& g(x)\quad\text{for all }x\in\rr^d\,.
\end{eqnarray}
In this case the function $v$ solves the PIDE in the classical sense i.e. point wise.
Beyond that, in cases where a point wise solution may fail to exist, a Feynman-Kac formula ties together weak solutions of certain PIDEs and conditional expectations, see \cite {BensoussanLions}.

\section{Definition of the Sobolev index}

\noindent According to inequality \eqref{A<=pol2}, the symbol $A$ belongs to $S^0_2$, we have $A \uhut \in L^2(\rrd)\cap L^1(\rrd)$ for every function $u \in S(\rrd)$ and the Fourier inverse of $A \uhut$ is well defined. Moreover an elementary calculation shows that
\begin{align}\label{def-pdo}
\frac{1}{(2\pi)^d} \int A(\xi) \uhut(\xi) \ee{ - i \skl x, \xi \skr} \dd \xi  = \OA u(x) \qquad \text{for all }x\in \rrd \,\,\text{and all } u \in S(\rrd) \,.
\end{align}
\zitep{[page 18]PhdGlau}Equation \eqref{def-pdo} coincides with the definition of a \emph{pseudo differential operator $\OA$ with symbol $A\in S^0_2$}. In other words, we have checked that $A$ is indeed the symbol of the so called pseudo differential operator (PDO) $\OA$.
\begin{remark}
Let $L$ be a L\'evy process with infinitesimal generator $\OG\!$. Since the PDO $\OA = -\OG$ is real-valued the associated symbol $A$ satisfies
$$
A(\xi) = \overline{A(-\xi)}\qquad \text{for all }\xi\in\rrd\,.
$$
\end{remark}
\zitep{ for a detailed proof see \cite[p. 206]{PhdGlau}.}
In the sequel we will work with \emph{Sobolev-Slobodeckii spaces}. These are defined by
$$
H^s(\rr^d) = \big\{u\in S'(\rr^d)\,\big|\, \hat{u}\in L^1_{\loc}(\rr^d,\ccd) \quad\hbox{with } \|u\|_s^2<\infty \big\}
$$
for $s\in\rr$ with
$$
\|u\|_s^2 = \int \left|\hat{u}(\xi) \right|^2 \big(1+|\xi| \big)^{2s} \dd \xi \,,
$$
where $S'(\rrd)$ denotes the space of \emph{generalised functions} i.e. the dual space of the Schwartz space $S(\rrd)$.
\par
The following assertion is taken from \cite[Lemma 4.4]{Eskin}. To keep our presentation self contained we include the short but crucial proof.
\begin{lemma}\label{st-p.o.}
Let $A\in S^0_\alpha$ with PDO $\OA$. Then there exists a constant $C\ge 0$, such that
$$
\|\OA u\|_{s-\alpha} \le C\|u\|_s \quad\hbox{ for all } u\in S(\rrd)
$$
for every $s\in \rr$. Furthermore the operator
$
\OA: S(\rr^d) \rightarrow  C^\infty(\rrd,\cc)
$
has a unique linear and continuous extension
$$
\OA: H^s(\rr^d) \rightarrow  H^{s-\alpha}(\rr^d) \,.
$$
\end{lemma}
\begin{proof}
From the definition of the norm and since $A\in S^0_\alpha$, we conclude
\begin{eqnarray*}
\|\OA u\|_{s-\alpha}^2 
&=& 
\int (1+|\xi|)^{2(s-\alpha)}|A(\xi)\hat{u}(\xi)|^2 \dd \xi\\
&\le&
 C \int (1+|\xi|)^{2s}|\hat{u}(\xi)|^2 \dd \xi\\
&=& C \|u\|^2_s \,.
\end{eqnarray*}
\noindent Obviously $\OA u \in C^\infty$  holds for every $u\in S(\rrd)$ and since $S(\rrd)$ is dense in $H^s(\rr^d)$ there exists a unique linear and continuous extension $\OA: H^s(\rr^d) \rightarrow  H^{s-\alpha}(\rr^d)$.
\end{proof}

\noindent For each $s\in\rr$ the dual space $(H^s(\rrd))^\ast$ of the Sobolev-Slobodeckii space $H^s(\rrd)$ is  isomorphic to $H^{-s}(\rrd)$, compare \cite[S. 62, 63]{Eskin}. Together with Lemma \ref{st-p.o.} this leads to
\begin{proposition}\label{st-sybol-Hs}
If $\OA$ is a PDO with symbol $A\in S^0_\alpha$, then
$$
\OA : H^s(\rrd) \longrightarrow \big(H^s(\rrd)\big)^*
$$
is continuous for $s = \alpha/2$ and the associated bilinear form $a:H^s(\rrd) \times H^s(\rrd) \to \cc$ defined by
$$
a(u,v) := (\OA u)(v)
$$
is continuous on $H^s(\rrd)$ i.e. there exists a constant $c>0$ with
\begin{eqnarray*}
 \big| a(u,v) \big| \le c\|u\|_s \|v\|_s\qquad \text{for all } u,v\in H^s(\rrd)\,.
\end{eqnarray*}
\end{proposition}
\noindent Let us now observe that for a PDO $\OA$ with symbol $A\in S^0_\alpha$ and bilinear form $a$ we have 
\begin{equation}\label{bilform_via_symbol}
a(u,v) = \int (\OA u)(x) \overline{v(x)} \dd x = \int A(\xi)   \uhut(\xi) \overline{\vhut(\xi)} \dd\xi
\end{equation}
for every $u,\,v\in S(\rrd)$ by Parseval's identity.
Since the bilinear form is expressed in terms of the symbol, the coercivity and the G{\aa}rding inequality translate to properties of the symbol. We will study coercivity and G{\aa}rding inequality with respect to Sobolev-Slobodeckii spaces.
\par
Let $A\in S^0_{\alpha}$ and assume the existence of a positive constant $c_1$ with
\begin{eqnarray}\label{coerc-cond}
  \Re(A(\xi))
 \ge
 c_1(1+ |\xi|)^{\alpha}\qquad \text{ for all }\xi\in \rrd \,.
 \end{eqnarray}
Then for any $u\in S(\rrd)$
\begin{eqnarray*}
\Re\big( a(u,u) \big)
=
\int \Re\big(A(\xi)\big) |\uh(\xi)|^2 \dd \xi 
\ge
c_1 \int (1+|\xi|)^{\alpha} |\uh(\xi)|^2 \dd \xi
= 
c_1 \|u\|_{\alpha/2}^2 \,.
\end{eqnarray*}
With the density of $S(\rrd)$ in $H^{\alpha/2}(\rrd)$, the coercivity of the bilinear form $a$ with respect to the Hilbert space $H^{\alpha/2}(\rrd)$ follows. Hence, if the symbol $A\in S_\alpha^0$ of a L\'evy process $L$ satisfies the coercivity condition \eqref{coerc-cond}, the infinitesimal operator of $L$  is elliptic. Moreover the corresponding parabolic equation has a unique solution in the Sobolev-Slobodeckii space $H^{\alpha/2}(\rrd)$, which will be discussed in detail in Theorem \ref{exeind-anhandsymbol}.

Let us point out that in contrast to the usual assumptions on a symbol, compare e.g. estimate (B.2) in \cite{Jacob.III}, we do not require any order of differentiability of the symbol. It is well known that the natural domain of the pseudo differential operator $\OA$ is the 
 $\psi$-Bessel potential space
$$
H^{\psi,2}_p(\rrd) = \bigg\{u\in L^2(\rrd)\,\bigg| \int_{\rrd} \big(1+ \psi(\xi)\big)^2|\uhut(\xi)|^2\dd \xi <\infty\bigg\}
$$
for $\psi(\xi):=\theta(i\xi)=-A(-\xi)$, that is studied in detail in \cite{FarkasJacobSchilling2001}. We are equally interested in the ellipticity of the operator, hence we investigate also the G{\aa}rding inequality.

Notice that the real part of the symbol of a L\'evy process is nonnegative,
\begin{align}\label{reAge0}
\Re\big(A(\xi)\big)
=
\langle \xi, \sigma \xi\rangle - \int \big( \cos(\langle x, \xi\rangle)-1\big)F(\dd y)
\ge0.
\end{align}
It is straightforward to verify that the space $H^{\Re(A)}:=\overline{C_0^\infty(\rrd)^{\|\cdot\|_{\Re(A)}}}$, that is the completion of $C_0^\infty(\rrd,\rr)$ with respect to the norm $\|\cdot\|_{\Re(A)}$ given by
\begin{align*}
\|u\|_{\Re(A)}:= \int_{\rrd} \big(1+\Re\big(A(\xi)\big)\big)|\uhut(\xi)|^2\dd\xi,
\end{align*}
is a Hilbert space. Moreover, $H^{\Re(A)}\hookrightarrow L^2(\rrd)\hookrightarrow\big(H^{\Re(A)}\big)^\ast$ is a Gelfand triplet, where $\big(H^{\Re(A)}\big)^\ast$ denotes the dual space of $H^{\Re(A)}$.
For $u\in C_0^\infty(\rrd,\rr)$ it follows
\begin{align*}
a(u,u) = \int_{\rrd} A(\xi) |\uhut(\xi)|^2\dd \xi
 =
 \int_{\rrd} \Re\big(A(\xi)\big) |\uhut(\xi)|^2\dd \xi 
=
\|u\|^2_{\Re(A)} - \|u\|^2_{L^2}.
\end{align*}
If we assume
\begin{align}\label{assumeImle1+Re}
\big|\Im\big(A(\xi)\big)\big|
\le c\big(1+\Re\big(A(\xi)\big)\big)\qquad\text{for all }\xi\in\rrd
\end{align}
with some positive constant $c\ge0$, then we obtain for $u,v\in C_0^\infty(\rrd,\rr)$
\begin{eqnarray*}
 |a(u,v)|
&=&
\left| \int_{\rrd} A(\xi) \uh(\xi) \overline{\vh(\xi)} \dd \xi \right| \\
&\le&
\left| \int_{\rrd} \Re\big(A(\xi) \big)\uh(\xi) \overline{\vh(\xi)} \dd \xi \right| 
+ \left| \int_{\rrd} \Im\big(A(\xi) \big)\uh(\xi) \overline{\vh(\xi)} \dd \xi \right| \\
&\le&
(1+c) \int_{\rrd} \left|\big(1+\Re\big(A(\xi) \big) \big)\right| \left|\uh(\xi) \overline{\vh(\xi)}\right| \dd \xi \\
&\le&
(1+c)\|u\|_{\Re(A)} \|v\|_{\Re(A)}.
\end{eqnarray*}
From the classical result on existence and uniqueness of solutions of parabolic differential equations, compare e.g. \cite{Wloka-english}, we obtain the following result.
\begin{theorem}\label{abstracttheorem}
Let $A$ be the symbol of the L\'evy process $L$. Assume \eqref{assumeImle1+Re}.
Then, the bilinear form $a$ is continuous w.r.t. $H^{\Re(A)}$ and  satisfies a G{\aa}rding inequality w.r.t. $H^{\Re(A)},L^2(\rr^d)$. In particular, the PIDE 
\begin{eqnarray*}
 \dot{u} + \OA u &=& f \\
u(0) &=& g
\end{eqnarray*}
with $f\in L^2\big(0,T; \big( H^{\Re(A)}(\rr^d)\big)^* \big)$ and initial condition $g\in L^2(\rr^d)$ has a unique solution $u\in W^1\big( 0,T; H^{\Re(A)},L^2(\rr^d)\big)$.
\end{theorem}

For a given Gelfand triplet $V\hookrightarrow H\hookrightarrow V^\ast$, the space $W^1\big(0,T; V, H\big)$ consists of those functions $u\in L^2\big(0,T;V\big)$ that have a derivative $\partial_t{u}$ with respect to time in a distributional sense that belongs to the space $L^2\big(0,T; V^\ast\big)$. For a Hilbert space $H$, the space $L^2\big(0,T; H\big)$ denotes the space of functions $u:[0,T]\to H$, that are weakly measurable and that satisfy $\int_0^T\|u(t)\|_H^2 \dd t < \infty$. For the definition of weak measurability and for a detailed introduction of the space $W^1\big(0,T; V, H\big)$ that relies on the Bochner integral, we refer to the book of \cite{Wloka-english}.

In the following, we focus on the case that the space 
$H^{\Re(A)}$ is a Sobolev-Slobodeckii space $H^s(\rrd)$, i.e. to the case that the function $\Re(A)$ in the definition of the space $H^{\Re(A)}$ can be replaced by a polynomial $|\xi|^{\alpha}$ with $\alpha\in(0,1]$. One major advantage of these more concrete spaces is that the index of a Sobolev-Slobodeckii space indicates a certain degree of smoothness.
This leads us to define the Sobolev index of a PDO resp. of a L\'evy process.
\begin{defin}
Let $\OA$ be a PDO with symbol $A$. We say $\alpha\in (0,2]$ is the \emph{Sobolev index} of the symbol $A$, if for all $\xi\in \rrd$
\begin{eqnarray*}
\begin{array}{rcll}
 \big|A(\xi)\big|
&\le&
C_1\left(1+|\xi|^2\right)^{\alpha/2}\quad& \text{(Continuity condition) and }\\
\Re\big(A(\xi)\big)
&\ge&
C_2 |\xi|^\alpha - C_3\left( 1+|\xi|^2\right)^{\beta/2}\quad& \text{(G{\aa}rding condition)} 
\end{array}
\end{eqnarray*}
for some $0\le \beta<\alpha$ and constants $C_1,\,C_3\ge 0$ and $C_2>0$.

If $L$ is a L\'evy process with symbol $A$ and Sobolev index $\alpha$, we call $\alpha$ the \emph{Sobolev index of the L\'evy   process} $L$.
\end{defin}
Let us notice that the G{\aa}rding condition is an assumption on the asymptotic behaviour of the real part of the symbol for large values of $\xi$. In case of continuity of $\xi\to A(\xi)$, it is equivalent to the existence of a number $N>0$, such that 
$$
\Re\big(A(\xi)\big)
\ge
C_2 |\xi|^\alpha \qquad\text{for all }|\xi|>N\,.
$$

Not every L\'evy process has a Sobolev index, compare Example \ref{bsp-A-CGMY}. But for important classes of L\'evy processes we will show its existence in Section \ref{ab-sobordn} and \ref{sec-alpha-stable}.
\begin{proposition}\label{Gard-folgtDichte}
If the L\'evy process has Sobolev index $\alpha>0$, then for every $t>0$, the measure $\mu_t=P^{L_t}$ has a smooth and bounded density w.r.t. the Lebesgue measure.
 \end{proposition}
\begin{proof}
The Fourier transform of the measure $\mu_t$ is given by $\hat{\mu}_t(\xi) = \ee{-t A(-\xi)}$ and
$$
\left|\hat{\mu}_t(\xi) \right| =  \ee{-t \Re\left(A(-\xi)\right)}\le \ee{- C_2 t |\xi|^\alpha + C_3 t \left( 1+|\xi|^2\right)^{\beta/2} }
$$
with $C_2>0$, $C_3\ge0$ and $0\le \beta<\alpha$ by assumption. This shows that the term $\left|\hat{\mu}_t(\xi) \right|$ decays exponentially fast for $|\xi|\rightarrow \infty$. Together with the continuity of $\xi \mapsto \Re\left(A(\xi)\right)$ finiteness of the moments $\int_{\rrd} |\xi|^n\left|\hat{\mu}_t(\xi) \right| \dd \xi < \infty$ for every $n\in\nn$ follows. The assertion now follows from \citeN[Proposition 28.1]{Sato}.
\end{proof}

Proposition \ref{Gard-folgtDichte} shows that the existence of a Sobolev index indicates the smoothness of the distribution of the process. Together with Proposition \ref{satz-momenteundblumenthal}, the assertion establishes ties between the smoothness of the distribution and path properties of the process, see the comments below Remark \ref{rem-rel-sob-blum}.

\noindent Before proving that the G{\aa}rding condition on the symbol entails a G{\aa}rding inequality of the associated bilinear form, we derive an elementary inequality:\\[1ex]
For $C_1>0$, $C_2\ge 0$, $0\le \beta <\alpha$ and $0<C_3<C_1$ there exits a constant  $C_4>0$ such that
\begin{align}\label{elementary-ineq}
C_1 x^\alpha - C_2 x^\beta \ge C_3 x^\alpha - C_4 \quad\text{for all }x\ge 0\,.
\end{align}
To show inequality \eqref{elementary-ineq}, it is enough to realize that for given constants $C_1,C_2,C_3,\alpha$ and $\beta$ as above, the point $x_0 = \left( \frac{\beta C_2}{\alpha(C_1-C_3)}\right)^{1/(\alpha- \beta)}$ is a global minimum of the function $f(x):= (C_1-C_3) x^\alpha - C_2 x^\beta$ on $\rr_{\ge 0}$.

\begin{lemma}
 Let $A\in S^0_\alpha$. If there exist constants $C_2>0$, $C_3\ge0$ and $0\le \beta<\alpha$ with
\begin{eqnarray*}
 \Re(A(\xi))
\ge
C_2|\xi|^{\alpha} - C_3(1+|\xi|^2)^{\beta/2}\qquad(\xi\in \rrd)\,,
\end{eqnarray*}
then the corresponding bilinear form satisfies a G{\aa}rding inequality with respect to $H^{\alpha/2}(\rrd)\hookrightarrow L^{2}(\rrd)$, i.e. there exist constants $c_2>0$ and $c_3 \ge0$ with
\begin{eqnarray*}
\Re\big( a(u,u) \big)
\ge
c_2\|u\|^2_{\alpha/2} - c_3\|u\|^2_{L^2} \,.
\end{eqnarray*}
\end{lemma}
\begin{proof}
For $u\in S(\rrd)$ we have
\begin{eqnarray*}
 \Re\big( a(u,u) \big)
&\ge&
\int \left( C_2 |\xi|^{2\alpha} - C_3 (1+|\xi|^2)^{\beta } \right) |\uh(\xi)|^2 \dd \xi \,.
\end{eqnarray*}
Furthermore we have $(1+|x|^2)^{\beta } \le 2^{\beta } (1+|x|^{2\beta })$, since
for $f(x) = (1+|x|^{2\beta })$ and $g(x) = (1+|x|^2)^{\beta }$ we get
\begin{eqnarray*}
 \frac{2^{\beta } f(x) }{g(x)}
=
\frac{2^ {\beta }}{(1+|x|^2)^{\beta } }+ \left(\frac{2|x|^2}{1+x}\right)^{\beta } \,.
\end{eqnarray*}
The first summand is bigger or equal to $1$ if $x\le 1$, whereas the second summand is bigger or equal to $1$ if $x\ge 1$. As both summands are positive for $x\ge 0$, we have  $2^{\beta } f(x) \ge g(x)$ for all $x\ge 0$.
Together with inequality \eqref{elementary-ineq} this yields
\begin{eqnarray*}
 C_2 |\xi|^{2\alpha} - C_3 (1+|\xi|^2)^{\beta } 
&\ge C_2 |\xi|^{2\alpha} - C_3 '(1+|\xi|^{2\beta }) 
\ge 
c_2 (1+|\xi|)^{2\alpha} - c_3
\end{eqnarray*}
with a strictly positive positive constant $c_2$ and $C_3', c_3\ge0$,
which yields the result.

\end{proof}
As argued to conclude Theorem \ref{abstracttheorem}, from the classical result on existence and uniqueness of solutions of parabolic differential equations, we obtain the following result.
\begin{theorem}\label{exeind-anhandsymbol}
 Let $\OA$ be a PDO with symbol $A$ and Sobolev index $\alpha$ for some $\alpha>0$. Then the parabolic equation
\begin{equation}\label{parab-gln}
\begin{split}
 \partial_t u + \OA u =& f \\
u(0) =& g\,,
\end{split}
\end{equation}
for $f  \in L^2\big(0,T;H^{-\alpha/2}(\rrd)\big)$ and $g\in L^2(\rrd)$ has a unique weak solution $u$ in the space $W^1\big(0,T; H^{\alpha/2}(\rrd), L^2(\rrd)\big)$.
\end{theorem}
Moreover the solution $u$ depends continuously on the data $g$ and $f$. The proof of the classical theorem is based on a so-called Galerkin-approximation that yields a numerical scheme to calculate the solution approximately, namely a finite element scheme, see e.g. \cite[Theorem 23.A]{Zeidler}.
%

In light of Theorem \ref{exeind-anhandsymbol}, the Sobolev index appears as a measure of the degree of the smoothing effect of the related evolution problem.
Under appropriate additional assumptions, a Feynman-Kac formula for weak solutions yields a stochastic representation.
Thus, the Sobolev index represents a measure for the smoothing effect of the distribution of the L\'evy process.
\section{Sobolev indices of L\'evy processes}\label{ab-sobordn}
Let us observe that for two L\'evy processes $L^i$ with symbol $A^i$ and Sobolev index $\alpha_i$ for $i=1,\,2$, the sum $L:=L^1+L^2$ is a L\'evy process with symbol  given by $A:=A^1+A^2$, and obviously the process has a Sobolev index that equals $\max(\alpha_1,\alpha_2)$.

\begin{example}[L\'evy process with Brownian part]
 $\rrd$-valued L\'evy pro\-cesses $L$ with characteristics $(b,\sigma,F)$ with a positive definite matrix $\sigma$ have Sobolev index $2$.
\end{example}
\begin{proof}
Let us observe that
$$
\Re\big(A(\xi)\big) = \frac{1}{2} \skl \xi, \sigma \xi \skr + \int \Big( 1 - \cos\big(\skl \xi, h(y)\skr \big)\Big) F(\dd y) \ge \frac{1}{2} \skl \xi, \sigma \xi \skr\,.
$$
 Since the matrix $\sigma$ is symmetric and positive definite $\underline{\sigma} |\xi|^2 \le \langle\xi,\sigma\,\xi\rangle$ for all $\xi\in\rrd$, 
 where $0<\underline{\sigma}$ is the smallest eigenvalue of the matrix $\sigma$. As a consequence we have $\underline{\sigma}|\xi|^2 \le \Re\big(A(\xi)\big)$, i.e. the G{\aa}rding condition. Continuity follows immediately from inequality \eqref{A<=pol2}.
\end{proof}
\begin{example}\label{ex-multidNIG}(Multivariate NIG-processes)
Let $L$ be an $\rr^d\!$-valued NIG-process, i.e.
 \[L_1=(L^1_1,\ldots,L^d_1)\sim\text{NIG}_d(\alpha,\beta,\delta,\mu,\Delta),\] 
with parameters
$\alpha,\delta\in\rr_{\geqslant0}$, $\beta,\mu\in\rr^d$ and a symmetric positive definite matrix $\Delta\in\rr^{d\times d}$ with $\alpha^2>\langle \beta, \Delta\beta \rangle$. Then the characteristic function of $L_1$ in $u\in\rr^d$ is given by
\begin{align*}
E\!\ee{i\langle u,L_1\rangle}
 \!&= \exp\!\left(\! i \langle u,\mu\rangle 
       +  \delta\!\Big(\sqrt{\alpha^2-\langle \beta,\Delta\beta\rangle}
           -\sqrt{\alpha^2-\langle \beta+iu,\Delta(\beta+iu)\rangle}\Big)\right)\!,
\end{align*}
where by $\skl \cdot,\cdot\skr$ we denote the product $\skl z,z'\skr = \sum_{j=1}^d z_j z_j'$ for $z\in\ccd$. Note that this is not the Hermitian scalar product. In \cite{Barndorff-Nielsen1977} multivariate NIG-distributions are derived as a subclass of multivariate GH-distributions via a mean variance mixture. 
We verify that $\rrd$-valued NIG-processes have Sobolev index $1$.
\end{example}

\begin{proof}
Similar to the calculations in \cite[Appendix~B]{EberleinGlauPapapantoleon10a} for real-valued NIG-processes,
\begin{align*}
z
:=&\, \alpha^2 - \skl \beta - iu, \Delta(\beta - iu)\skr \\
=&\,
\alpha^2 - \skl \beta, \Delta \beta\skr + \skl u, \Delta u\skr +  i \skl \beta, \Delta u\skr + i \skl u, \Delta\beta\skr
\end{align*}
and $\sqrt{z} = \sqrt{\frac{1}{2}(|z| + \Re(z))} + i \frac{\Im(z)}{|\Im(z)|} \sqrt{\frac{1}{2}(|z| - \Re(z))}$
it follows
$
|z| \ge \alpha^2 - \skl \beta, \Delta \beta\skr + \skl u, \Delta u\skr > 0
$
and
\begin{align*}
 \Re\big(A(u)\big)
=&\,
-\delta \sqrt{\alpha^2 - \skl\beta,\Delta \beta\skr } + \delta \Re\big( \sqrt{ z}\big) \\
=&\,
\frac{\delta}{\sqrt{2}} \sqrt{|z| + \Re(z)}
- \delta \sqrt{\alpha^2 - \skl\beta,\Delta \beta\skr }  \\
\ge&\,
  \delta \sqrt{\alpha^2 - \skl\beta,\Delta \beta\skr + \skl u, \Delta u \skr} 
- \delta \sqrt{\alpha^2 - \skl \beta, \Delta \beta \skr }\\
\ge&\,
\delta \sqrt{\lambda_{\min} } |u| - \delta \sqrt{\alpha^2 - \skl\beta,\Delta \beta\skr} \,,
\end{align*}
where $\lambda_{\min}$ denotes the smallest eigenvalue of the matrix $\Delta$. Analogously it follows that
$|\Re(u) |\le C_1(1+|u|)$ and $|\Im(u) |\le C_2(1+|u|)$ with positive constants $C_1,\,C_2$, which yields
$|A(u) |\le C(1+|u|)$ with a positive constant $C$.
\end{proof}
\begin{example}[Cauchy processes]
Let $L$ be a Cauchy process with values in $\rrd$, then the distribution $\mu:=P^{L_1}$ has the Lebesgue density
$$
f(x) = c \frac{\Gamma((d+1)/2)}{\pi^{(d+1)/2}} \big(|x-\gamma|^2 + c^2\big)^{-(d+1)/2}
$$
where $\mu\in\rr^d$ and $c>0$, and its characteristic function is given in 
$$
\muhut(u) = \ee{-c|u| + i\langle \gamma,u \rangle }\,,
$$
see \cite[Example 2.12]{Sato}. It follows immediately that the process has Sobolev index $1$.

\end{example}
\begin{example}[Student-$t$ processes]
Let $L$ be a L\'evy process such that the distribution of $L_1$ is student-$t$ with parameters $\mu\in\rr$, $f>0$ and $\delta>0$ i.e.
$$
P^{L_1}(\dd x) = \frac{\Gamma\big((f+1)/2\big)}{\sqrt{\pi\delta^2} \, \Gamma(f/2)} \Big( 1 + \frac{x-\mu}{\delta^2} \Big)^{-(f+1)/2}\,.
$$
This generalisation of the student-$t$ distribution is studied in \cite{EberleinHammerstein04}, where it appears as limit of GH distributions for parameters $\alpha,\,\beta \downarrow 0$ with negative $\lambda$.

We show that $L$ has Sobolev index $1$.
\end{example}
\noindent
\par
\begin{proof}
The characteristic function $\muhut$ of the student-$t$ distribution reads as follows
$$
\muhut(u) = \left(\frac{f}{4}\right)^{f/4} \frac{2 K_{-f/4}\left(\sqrt{f}|u|\right)}{\Gamma\big(f/2\big)} |u|^{f/4} \ee{i\mu u}\,.
$$
with $\delta:=f/4$ and $c:=\log\left\{(f/4)^{f/4}/\Gamma(f/2) \right\}$, compare \cite{EberleinHammerstein04}. We obtain the following representation of the associated symbol,
\begin{equation}\label{eq_AwithKdelta}
A(u) 
=
 -c - \log\left\{ K_{-\delta} \big(2\sqrt{\delta} |u| \big) \right\} - \log\left\{|u|^{2\delta} \right\} + i \mu u \,.
\end{equation}
Since the mapping $u\mapsto A(u)$ is continuous, it is enough to verify the continuity and G{\aa}rding inequality for a function that is asymptotically equivalent to $A$. To this aim we insert the asymptotic expansion of the Bessel function $K_\lambda$, see \cite[equation  (9.7.2)]{AbramowitzStegun}\zitep{S.378}.
\begin{align*}
K_\lambda(z) \sim& \sqrt{\frac{\pi}{2z}} \ee{-z} \bigg\{ 1+ \frac{\mu-1}{8z} + \frac{(\mu-1)(\mu-9)}{2!(8z)^2} \\
&+ \frac{(\mu-1)(\mu - 9)(\mu-25)}{3!(8z)^3} + \cdots\bigg\}
\end{align*}
for $|\arg z|<\frac{3}{2} \pi$ and $|z|\rightarrow \infty$ with $\mu = 4\lambda^2$ with the usual notation $f(x)\sim g(x)$ for $|x|\to\infty$ if $\frac{f(x)}{g(x)}\to 1$ for $|x|\to \infty$. \\
In particular
$$
K_\lambda(z) \sim  \sqrt{\frac{\pi}{2z}} \ee{-z}=:g(z)\quad\text{for $z$ real and }z\rightarrow \infty
$$
and $K_\lambda(z) \rightarrow 0$ as well as $g(z)\rightarrow 0$ for $z\rightarrow\infty$. It follows
$$
\frac{\log\big(K_\lambda(z)\big)}{\log\big(g(z)\big)}\sim \frac{g'(z)}{K_\lambda'(z)} \quad\,\text{with }\,K_\lambda'(z) =  \frac{\lambda}{z}K_\lambda(z) - K_{\lambda+1}(z)\,,
$$
compare p. 79 equation (4) in \cite{Watson}. We conclude
$$
\frac{\log\big(K_{-\delta}(z)\big)}{\log\big(g(z)\big)}\sim \frac{ -g(z) - \frac{1}{2z} g(z) }{ \frac{{-\delta}}{z}K_{-\delta}(z) - K_{{-\delta}+1}(z)} \sim \frac{ g(z)}{K_{{-\delta}}(z) } \sim 1\,.
$$
Therefore
\begin{equation}\label{eq-logKdelta}
\log\left\{ K_{-\delta}\big(2\sqrt{\delta} |u| \big) \right\} \sim \log\left\{ \sqrt{\frac{\pi}{4\sqrt{\delta}|u|}} \right\} - 2\sqrt{\delta} |u|  + \log\left\{1+ \left|O\left(\frac{1}{|u|}\right)\right|\right\}
\end{equation}
for $|u|\rightarrow \infty$, where $O$ denotes Landau's symbol, i.e. we write $f(x) = O\big(g(x)\big)$ for $|x|\to\infty$ if there exists constants $M, N$ s.t. $\frac{|f(x)|}{|g(x)|}\le M$ for all $|x|>N$. Inserting equation \eqref{eq-logKdelta} in equation \eqref{eq_AwithKdelta} we obtain for the real part of the symbol
\begin{eqnarray*}
\Re\big( A(u)\big) 
&\sim& 
-c-  \log\left\{ \sqrt{\frac{\pi}{4\sqrt{\delta}|u|}} \right\} + 2\sqrt{\delta} |u|\\
&&  - \log\left\{1+ \left|O\left(\frac{1}{|u|}\right)\right|\right\} - 2\delta\log{|u|}\,.
\end{eqnarray*}
From the boundedness of the term $\log\left\{1+ \left|O\left(\frac{1}{|u|}\right)\right|\right\}$ for $|u|\to \infty$ we further get
\begin{eqnarray*}
\Re\big( A(u)\big) 
&\sim& 
 2\sqrt{\delta}|u| - \big(1+2\delta\big) \log{|u|} \\
&\ge&
 2\sqrt{\delta}|u| - \big(1+2\delta\big) \sqrt{|u|}\,,
\end{eqnarray*}
since $|\log{|u|}| \le \sqrt{|u|}$. This shows the G{\aa}rding-condition and moreover that $\big|\Re\big( A(u)\big)\big| \le c|u|$ for some constant $c\ge 0$. Furthermore the imaginary part equals $\Im\big(A(u)\big) = \mu u$, hence the continuity condition is also satisfied.
\end{proof}

\subsection{Sobolev index for L\'evy processes with absolutely continuous L\'evy measure}\label{sec-abscontF}

In this subsection, we study the Sobolev index for real-valued L\'evy processes without Brownian part whose L\'evy measure has a Lebesgue density.
If the process has no Brownian part, the G{\aa}rding condition only depends on the real part of the  integral $\int \big(\ee{-iux} - 1 + ih(x)u\big) F(\dd x)$, which translates to properties of the symmetric part of the L\'evy measure.
\par
Let $A$ be the symbol of a real-valued L\'evy process that is a special semimartingale $L$ with operator $\OA$. Let $(b,0,F)$ be the characteristic triplet of $L$ w.r.t. $h(x) = x$. Furthermore assume $F(\dd x) = f(x) \dd x$ for the L\'evy measure $F$. We denote by $f_s$ the symmetric and by $f_{as}$ the antisymmetric part of the density function $f$, i.e. $f_s(x) = \frac{1}{2} \big( f(x) + f(-x)\big)$ and $ f(x)= f_s(x) + f_{as}(x)$ for every $x\in\rr$.

For every $u\in\rr$ we define
\begin{eqnarray*}
 A^{f_s}(u) &:=& - \int \left( \ee{-iux} - 1 + iux\right) f_s(x) \dd x 
\,=\, 
- \int \left( \cos(ux) - 1\right) f_s(x) \dd x, \\
A^{f_{as}}(u) &:=& - \int \left( \ee{-iux} - 1 + iux\right) f_{as}(x) \dd x
\,=\,
 i \int \left( \sin(ux) - ux \right) f_{as}(x) \dd x,\\
A^{f}(u) &:=& - \int \left( \ee{-iux} - 1 + iux\right) f(x) \dd x
 \,=\,
A^{f_s}(u) + A^{f_{as}}(u) .
\end{eqnarray*}
Note the following equalities,
\begin{align}
\Re(A^f) &= A^{f_s}\,,\quad i\Im(A^f) = A^{f_{as}}\,,\label{eq-reA-imA}\\
\Re\big(A(u)\big) &= \frac{1}{2} \skl u, \sigma u\skr + A^{f_s}(u)\,,\label{eq-reA}\\
\Im\big(A(u)\big) &= \skl u, b\skr  - iA^{f_{as}}(u) = \skl u, b\skr + \int \big( \sin(ux) - ux\big) f_{as}(x) \dd x \,. \label{eq-imA}
\end{align}
Let us further notice the following elementary assertion.
\begin{lemma}\label{bem_fas-kl-fs}
If $F$  is a nonnegative measure, absolutely continuous with respect to the Lebesgue measure, it has a nonnegative density $f\ge 0$ and $|f_{as}| \le f_s$.
\end{lemma}
\begin{proof}
If $F(\dd x) = f(x) \dd x$ with a nonnegative measure $F$, then $F(\dd x) = |f(x)|\dd x$ hence w.l.g. $f(x) \ge 0$.

Thus we have $f_s\ge 0$, since otherwise there would exist a number $y\in\rr$ with $f_s(y)<0$ and $f_s(-y)=f_s(y)<0$.
However, since $f_{as}(y)\le 0$ or $f_{as}(-y)\le 0$ we would get $f(y)<0$ or $f(-y)<0$ i.e. a contradiction.

Furthermore we have $|f_{as}|\le f_s$, since otherwise there would exist a number $y\in\rr$ with $-f_{as}(y)>f_s(y)$ i.e. $f(y) <0$ or $f_{as}(y) >f_s(y)$, from where we would get $f(-y)<0$.
\end{proof}

In the following proposition we derive the Sobolev index for L\'evy processes without Brownian part from the behaviour of the L\'evy measure $F$ around the origin. It uses a rather technical lemma that is provided in appendix A.

\begin{proposition}\label{satz-hilfF-sobindex}
Let $L$ be a real-valued L\'evy process and a special semimartingale with characteristic triplet $(b,0,F)$ with respect to the truncation function $h(x)=x$.

Let
\vspace*{-1ex}
\begin{eqnarray}\label{eq_f_s}
f_{s}(x) &= \frac{C}{|x|^{1+Y}} + g(x) \qquad\text{with $g(x) = O\left(\frac{1}{|x|^{1+Y-\delta}}\right)$ for $x\to0$}
\end{eqnarray}
with $0<\delta$. In the following cases, the L\'evy process $L$ has Sobolev index $Y$.
\begin{itemize}
 \item [a)]
Let $0<Y<1$ and
\vspace*{-1ex}
\begin{eqnarray*}
f_{as}(x) = O\left(\frac{1}{|x|^{\alpha}}\right)\quad\text{for $x\to0$}
\end{eqnarray*}
with $\alpha\le 1 + Y$,
$\int \big|x f_{as}(x)\big| \dd x <\infty$, and moreover $b= \int x F(\dd x)$.
\item[b)]
Let $Y=1$ and
\begin{eqnarray*}
f_{as}(x)
=
O\left(\frac{1}{|x|^{\alpha}}\right)\quad\text{for $x\to0$ }
\end{eqnarray*}
with $\alpha < 1+Y=2$.
\item[c)]
Let $1<Y<2$.
\end{itemize}
\end{proposition}
\begin{proof}
In each of the three cases, according to part b) of Lemma \ref{lem_sob_ordn}, the G{\aa}rding condition follows directly from
$$
\Re\left(A(u)\right) = \Re\big(A^{f}(u)\big) = \Re\big(A^{f_s}(u)\big) \ge C|u|^Y - C_1\left(1 + |u|^{Y'} \right)
$$
with $C>0$ and $C_1\ge 0$ and $0<Y'<Y$.
\par
Splitting $A$ in its real and its imaginary part, assertion a) and d) of Lemma \ref{lem_sob_ordn} yield the continuity condition in case a) with index $Y$, since the assumption $f_{as} (x) = O \left(\frac{1}{|x| ^\alpha} \right)$ for some $\alpha \in(0, 1+Y]$ implies $f_{as}(x) = O\left(\frac{1}{|x|^{1+Y}} \right)$.
\par
In order to verify the continuity conditions for b) and c), we first notice that
\begin{eqnarray*}
 \big|A(u)\big|
&\le&
\left|A^{f_s}(u)\right| + \left| \Im \left(A^f(u)\right) \right| + |b| |u| 
\le
C_2\!\big( 1+|u|^Y\big) +  \left| A^{f_{as}}(u) \right| + |b| |u| 
\end{eqnarray*}
follows from Lemma \ref{lem_sob_ordn} a).
\par
Concerning case b), we notice that the  assumption on $f_{as}$ implies $f_{as}(x) = O\left(\frac{1}{|x|^{1+Y'}}\right)$ with some $0<Y'<1$. Lemma \ref{lem_sob_ordn} c) yields $\left|A^{f_{as}}(u)\right|  \le C_3\left(1+|u|\right)$. Together with Lemma \ref{lem_sob_ordn} a) this yields
$$
\big|A(u)\big| \le C_2\left( 1+|u|\right) + C_2'\left( 1+|u|^{1-\delta}\right) +  C_3\left(1+|u|\right)  + |b| |u| \le C\left( 1+|u|\right)
$$
with nonnegative constants $C_2,C_2',C_3$ and $C$. In other words we have shown the continuity condition for case b).
\par
Finally, to verify the continuity condition for $1<Y<2$, let us notice that from Lemma \ref{bem_fas-kl-fs} we know $\left|f_{as}\right| \le f_s$ so that
$$
\left|f_{as}(x)\right| = O\left(\frac{1}{|x|^{1+Y}}\right)
$$
for $|x|\rightarrow 0$. Due to Lemma \ref{lem_sob_ordn} c) we have $\left|A^{f_{as}}(u)\right| \le C_3\left( 1+ |u|^Y\right)$ and altogether we obtain
\begin{align*}
\big|A(u)\big| \le C_2\left( 1+|u|^Y\right) +  C_3\left(1+|u|^Y\right)  + |b| |u| \le C'\left( 1+|u|^Y\right) \,. 
\end{align*}
\end{proof}
\begin{example}\label{bsp-A-GH}
Generalised Hyperbolic (GH) processes have Sobolev index $1$.
\end{example}
\begin{proof}
The L\'evy measure $F^{GH}$ of a GH process has a Lebesgue density $F^{GH}(\dd x) = f^{GH}(x)\dd x$ with
$$
f^{GH}(x) = C_1 \frac{1}{x^2} + C_2\frac{1}{|x|} + C_3\frac{1}{x} + \frac{o(|x|)}{x^2}
$$
with $\frac{o(|x|)}{|x|}\rightarrow 0$ for $|x|\rightarrow 0$, see \citeN[Proposition 2.18]{Raible}.
Hence the symmetric part of $f^{GH}$ is of the form
$$
f_s^{GH}(x) = \frac{C}{|x|^2} + O\left(\frac{1}{|x|}\right)\qquad\text{ for $x\to 0$,}
$$
and the antisymmetric part is of the form
$$
f_{as}^{GH}(x) = O\left(\frac{1}{|x|}\right)\qquad \text{ for $x\to 0$.}
$$
The assertion follows from part b) of Theorem \ref{satz-hilfF-sobindex}.
\end{proof}

\begin{example}\label{bsp-A-CGMY}
A CGMY L\'evy process with parameters $C$, $G$, $M>0$ and $Y<2$, is a L\'evy process that has no Brownian part and its L\'evy measure $F^{CGMY}$ is given by its Lebesgue density
\begin{eqnarray*}
f^{\operatorname{CGMY}}(x) = 
 \left\{
 \begin{array}{ll}
\frac{C}{|x|^{1+Y}}  \ee{G x}  \quad&\text{for }x<0 \\
\frac{C}{|x|^{1+Y}}  \ee{-Mx}  \quad&\text{for }x\ge0 \,,
 \end{array}
\right.
\end{eqnarray*}
compare \cite{CarrGemanMadanYor2002}.
\notiz{ Note that this definition coincides with the original definition in \cite{CarrGemanMadanYor2002} except for the choice of the truncation function resp. the drift parameter.}
\begin{itemize}
 \item[(i)] A CGMY L\'evy process with parameters $C$, $G$, $M>0$ and $Y\in(0,1)$ and characteristics $\big(Y(M^{Y-1} - G^{Y-1}), 0, F^{CGMY})$ with respect to the truncation function $h(x)=x$ has Sobolev index $Y$.
 \item[(ii)] A CGMY L\'evy process with parameters $C$, $G$, $M>0$ and $Y\in[1,2)$ has Sobolev index $Y$.
\end{itemize}

\end{example}
\par
\begin{proof}
For $Y\in(0,1)$, the assertion follows immediately from the explicit formula of the characteristic exponent of the distribution. Namely for a CGMY process $L$ with characteristics $\big(0,0,F^{CGMY})$ w.r.t. the truncation function $h(x) = x$  we have
\begin{align}\label{CGMY-cumulant}
\begin{split}
\log\big(E\ee{iu L_1}\big) =& \,\,C\Gamma(-Y) \big\{ (M-iu)^Y - M^Y +  Y(M^{Y-1} - G^{Y-1}) iu \\
&+ (G + iu)^Y - G^Y\big\}\,,
\end{split}
\end{align}
where $\Gamma$ denotes the analytic extension of the Gamma function, see \cite{PoirotTankov06}. 
\notiz{Formel und Beweis in \citeN[Theorem 1]{CarrGemanMadanYor2002} sind falsch! Cite Cont Tankov? Sonst Poirot and Tankov Monte Carlo option pricing for tempered stable (CGMY) processes}

For $Y\ge 1$ no explicit formula is available, we therefore examine the density of the L\'evy measure. The following decomposition in a symmetric and an antisymmetric part of the density is valid for any $Y\in(0,2)$,
\notiz{\begin{eqnarray*}
f^{\operatorname{CGMY}}_s(x) = \frac{1}{2} \big( f(x) + f(-x) \big) =
 \left\{
\begin{array}{ll}
\frac{C}{2} \frac{\ee{Gx} + \ee{Mx}}{|x|^{1+Y}} \qquad& \qquad\text{for }x<0 \,,\\[1ex]
\frac{C}{2} \frac{\ee{-Gx} + \ee{-Mx}}{|x|^{1+Y}} \qquad& \qquad\text{for }x\ge 0 \,,
\end{array}
\right.
\end{eqnarray*}
and so it follows}
\begin{eqnarray*}
f^{\operatorname{CGMY}}_s(x) &=& \frac{C}{2} \frac{\ee{-G|x|} + \ee{-M|x|}}{|x|^{1+Y}} = \frac{C}{|x|^{1+Y}} -  C\frac{2-\ee{-G|x|} -\ee{-M|x|}}{2 |x|^{1+Y}} \\
&=&
\frac{C}{|x|^{1+Y}}+
 O\left(\frac{1}{|x|^Y}\right)\quad\text{for }|x|\rightarrow 0\,.
\end{eqnarray*}
Furthermore we have
\notiz{\begin{eqnarray*}
f^{\operatorname{CGMY}}_{as}(x) = \frac{1}{2} \big( f(x) - f(-x) \big) =
 \left\{
\begin{array}{ll}
\frac{C}{2} \frac{\ee{Gx} - \ee{Mx}}{|x|^{1+Y}} \quad& \qquad\text{for }x<0 \,,\\[1ex]
\frac{C}{2} \frac{- \ee{-Gx} + \ee{-Mx}}{|x|^{1+Y}} \quad&\qquad\text{for }x\ge 0 \,,
\end{array}
\right.
\end{eqnarray*}
hence}
\begin{eqnarray*}
 \left|f^{\operatorname{CGMY}}_{as}(x) \right|
&=&
\frac{C}{2} \frac{\left| \ee{-G|x|} - \ee{-M|x|} \right|}{|x|^{1+Y}}
\,=\,
O\left(\frac{1}{|x|^Y}\right)\quad\text{for }|x|\rightarrow 0\,.
\end{eqnarray*}
For $Y\ge1$ we obtain from Proposition \ref{satz-hilfF-sobindex} b) and c) that $L$ has Sobolev index $Y$.
\end{proof}
We conclude this section with the following observation.
\begin{remark}
 Variance gamma (VG) processes are CGMY processes with parameter $Y=0$, and they do not have a Sobolev index.
\end{remark}

\section{Sobolev index of \boldmath{$\alpha$}-semi-stable L\'evy processes}\label{sec-alpha-stable}

Remember that for an $\alpha$-semi-stable L\'evy process $L$ there exists a deterministic function $t\mapsto c(t)$ and $a\!> \!1$ such that
$\left(L_{at}\right)_{t\ge 0}$ coincides with $\left(a^{1/\alpha}L_{t} + c(t)\right)_{t\ge 0}$ in distribution, compare Section 13 in \cite{Sato}.
\notiz{(To give a more precise reference, the assertion follows putting together Definition 13.16, Theorem 13.15, Theorem 13.11, Definition 13.1, Proposition 13.5 and Definition 13.4 in \cite{Sato}.)}

The symbol of a generic real-valued strictly $\alpha$-stable L\'evy process is of the form $A(u)=c|u|^\alpha$ with a constant $c>0$, see \cite[Theorem 14.9]{Sato}. In this case the L\'evy process obviously has Sobolev index $\alpha$.

In this section we show that any $\alpha$-semi-stable L\'evy process with $1<\alpha\le 2$ has Sobolev index $\alpha$. For $\alpha$-semi-stable L\'evy processes with $0<\alpha<1$ we give additional sufficient conditions under which the processes have Sobolev index $\alpha$. Additionally, it turns out that any real-valued strictly $\alpha$-stable L\'evy process has Sobolev index $\alpha$.

Let us give a definition of (semi) stability and $\alpha$-(semi) stability of L\'evy processes in terms of the symbol of the process according to Definition 13.1, Proposition 13.5, Definition 13.16, and Theorem 13.11 in \cite{Sato}. 
\begin{defin}\label{def_stabil}
A L\'evy process with symbol $A$ is called \emph{(semi) stable} if for any $0<a\neq 1$ (for some $0<a\neq 1$) there exists a constant $b>0$ and a vector $c\in\rr^d$ with
\begin{eqnarray}\label{eqn_semistabil}
a A(u) = A( bu) + i\langle c,u\rangle\,\quad\text{for all }u\in\rr^d\,.
\end{eqnarray}
A L\'evy process is called  \emph{$\alpha$-semi-stable}, if it is semi-stable and if
\begin{equation}
a=b^\alpha\quad\text{for all }a\in\Gamma
\end{equation}
with
$\Gamma := \big\{a>0\,\big|\, \exists b>0\,,\, c\in\rr^d\,\text{ s.t. \eqref{eqn_semistabil} is satisfied with $a$, $b$ and $c$}  \,\big\}$.
Accordingly, a L\'evy process is called \emph{$\alpha$-stable}, if it is stable and if
\begin{equation}
a=b^\alpha\quad\text{for all }a\in\Gamma=(0,\infty)\,.
\end{equation}
If $c=0$ in equality \eqref{eqn_semistabil}, the process is called \emph{strictly semi-stable}, \emph{strictly stable} resp. \emph{strictly $\alpha$-(semi) stable}.
\end{defin}
%
%
%
From Definition 13.16, Theorem 13.15 and from Theorem 14.1 in \cite{Sato} we obtain the following remark.
\begin{remark}\label{rem-chartriplet-alpha-stable}
Let $L$ be an $\alpha$-semi-stable L\'evy process with characteristic triplet $(b,\sigma,F)$, with $\sigma\neq 0$ or $F\not\equiv 0$.
 \begin{itemize}
  \item[a)]
We have $0<\alpha\le 2$.
 \item[b)]
We have $\alpha=2$ iff $\sigma\neq 0$ and $F\equiv0$.
\end{itemize}
\end{remark}
Before focusing on the Sobolev index for $\alpha$-semi stable L\'evy processes, we briefly discuss the notion of (non-)degeneracy of L\'evy processes.

According to Definition 24.16 and 24.18 in \cite{Sato}, an $\rrd$-valued L\'evy process $L$ is called \emph{degenerate}, if $P^{L_t}$ is degenerate for any (or equivalently for some) $t>0$, i.e.
\[
S_{P^{L_t}}= \big\{x\in\rrd\big|\, P^{L_t}(G)>0\,\text{for every open subset }G\in\rrd\,\text{with }x\in G \big\}
\]
is contained in some affine subspace of $\rrd$, i.e.
\[
S_{P^{L_t}}\subset y + V
\]
for some $y\in\rrd$ and some linear $(d-1)$-dimensional subspace of $\rrd$.

A L\'evy process that is not degenerate is said to be \emph{nondegenerate}. Note that the definition of degeneracy implies that non-constant real-valued L\'evy processes are nondegenerate. Proposition 24.17 (ii) shows that an $\rrd$-valued L\'evy process is nondegenerate, if its L\'evy measure is nondegenerate or if $\sigma(\rrd)=\{\sigma x|\,x\in\rrd\}$ is not contained in a $(d-1)$-dimensional linear subspace of $\rrd$.

From \cite[Proposition 24.20]{Sato}, the following relation between the Sobolev index and $\alpha$-stability can be deduced.
\begin{proposition}\label{stabil_garding}
Every nondegenerate $\alpha$-semi-stable L\'evy process satisfies the G{\aa}rding condition with index $\alpha$.
\end{proposition}
\begin{proof}
The assertion follows directly from Proposition 24.20 in \cite{Sato}.
\end{proof}

In view of Proposition \ref{stabil_garding},
 it is enough to study the continuity condition in the sequel.
The following proposition characterises the Sobolev index for $\alpha$-semi-stable L\'evy processes with $\alpha\neq 1$ and for real-valued $1$-stable L\'evy processes.
\begin{proposition}
Let $L$ be a nondegenerate $\alpha$-semi-stable L\'evy process.
\begin{itemize}
 \item[a)]
If $0<\alpha<1$, then $L$ has Sobolev index $\alpha$ iff $L$ is strictly $\alpha$-semi-stable.
\item[b)]
If $1< \alpha\le 2$, then $L$ has Sobolev index $\alpha$.
\item[c)]
The symbol of a real-valued strictly $\alpha$-stable L\'evy process with $\alpha=1$ and Sobolev index $1$ is of the form
$$
A(u) = c|u| + i\tau u
$$
with $c>0$ and $\tau\in\rr$.
\item[d)]
If the process $L$ is real-valued and $\alpha$-stable with $\alpha=1$, then $L$ has Sobolev index $1$ iff $L$ is strictly $1$-stable.
\end{itemize}
\end{proposition}
\begin{proof}
 In view of Remark \ref{rem-chartriplet-alpha-stable} the assertion is obvious for $\alpha=2$.
\par
 For $0<\alpha<2$ with $\alpha\neq 1$, Proposition 14.9 in \citeN{Sato} shows that the symbol $A=-\log(\muhut)$ is of the form
$$
A(u) = |u|^\alpha \big( \eta(u) + i\gamma_\alpha(u)\big) + i \langle c_\alpha,u\rangle
$$
 with $c_\alpha\in \rr^d$, $u\mapsto \eta(u)$ nonnegative, continuous on $\rrd\setminus \{0\}$ and $\eta(bu) = \eta(u)$ for all $u\in\rr^d$, and $\gamma_\alpha$ real-valued, continuous on $\rr^d\setminus \{0\}$ with $\gamma_\alpha(bu) = \gamma_\alpha(u)$ for all $u\in\rr^d$ with $b=a^{1/\alpha}>1$. Basic arguments show that the mappings $u\mapsto \eta(u)$ and $u\mapsto \gamma_\alpha(u)$ are bounded.
 %
We therefore have $\Re\left( A(u)\right) = |u|^\alpha \eta(u)$, where $\eta$ is bounded, and hence $\left|\Re\left( A(u)\right)\right| \le C |u|^\alpha$. In view of Proposition \ref{stabil_garding} it remains to derive an adequate upper bound of the imaginary part.
\par
For $0<\alpha<2$, $\alpha\neq1$ we have
$$
\Im\big(A(u)\big) = |u|^\alpha \gamma_\alpha(u) + \langle c_\alpha, u\rangle
$$
with the bounded function $\gamma_\alpha$.

For $0<\alpha<1$, this shows that $\left|\Im\big(A(u)\big) \right|\le C\left(1+|u|^\alpha\right)$ iff $c_\alpha = 0$. 
According to \cite[Theorem 14.7 (i)]{Sato} the latter is the case if and only if the distribution resp. the L\'evy process is strictly $\alpha$-semi-stable.
\par
For $1<\alpha<2$, due to
$$
\big|\langle c_\alpha, u\rangle\big| \le |c_\alpha||u| \le |c_\alpha|\big(1+|u|^\alpha\big),
$$
we obtain  $\left|\Im\big(A(u)\big) \right|\le C\left(1+|u|^\alpha\right)$ without further restrictions.
\par 
Assertion c) and d) for $\alpha=1$ are a direct consequence of Theorem 14.15, equation (14.25) in \cite{Sato}, that states that the symbol of a real-valued non-trivial (i.e. a non-constant) $1$-stable L\'evy process is
 the form
$$
A(u) = c|u| \left( 1- i\beta\frac{2}{\pi} \frac{u}{|u|} \log|u| \right) + i\tau u
$$
with $c>0$, $\beta\in[-1,1]$ and $\tau\in\rr$\label{A-stabiler-1d-Levy}.
From this representation of the symbol we can read that $L$ is strictly $1$-stable iff $\beta=0$. The representation given in assertion c) follows as well.
\end{proof}

\section{Connections with the Blumenthal-Getoor index}

The index $\beta$, called Blumenthal-Getoor index, quantifies the intensity of small jumps of a L\'evy process. It is defined for every L\'evy process, whereas not every L\'evy process has a Sobolev index. In this section we show for real-valued L\'evy processes that if they have a Sobolev index $Y<2$, then this index is bigger or equal to the Blumenthal-Getoor index.

The following definition of the Blumenthal-Getoor index is taken from \citeN[p. 362]{Sato}.
\begin{defin}\label{blumenindex}
Let $L$ be a L\'evy process with characteristics $(b,c,F)$. Then
\begin{eqnarray*}
 \beta := \inf\left\{ \alpha>0 \,\Big|\,\int_{[-1,1]} |x|^\alpha F(\dd x) <\infty  \right\}
\end{eqnarray*}
is called the \emph{Blumenthal-Getoor index} of the process.
\end{defin}
\noindent It is well known that the Blumenthal-Getoor index is related to path properties of the L\'evy process.  Theorem 21.9 in \citeN{Sato}\zitep{see also Definition 11.9} shows the following relationship between the Blumenthal-Getoor index and the variation of the paths of the L\'evy process.
\begin{proposition}
Let $L$ be a L\'evy process without Brownian part with characteristics $(b,0,F)$ and  Blumenthal-Getoor index $\beta$.
\begin{itemize}
  \item [(a)]
If $\beta<1$, then $P$-a.e. path of $L$ is of bounded variation on $(0,t]$ for every $t>0$.
\item[(b)]
If $\beta>1$, then $P$-a.e. path of $L$ is of unbounded variation on $(0,t]$ for every $t>0$.
\item[(c)]
If $\beta=1$, then we have the following two cases.
\begin{itemize}
\item[(c1)]
If $\int_{-1}^1 |x|F(\dd x) <\infty$, then $P$-a.e. path of $L$ is of bounded variation on $(0,t]$ for every $t>0$.
\item[(c2)]
If $\int_{-1}^1 |x|F(\dd x) = \infty$, then $P$-a.e. path of $L$ is of unbounded variation on $(0,t]$ for every $t>0$.
 \end{itemize}
\end{itemize}
\end{proposition}
\noindent In \citeN{HudsonMason76} and the references therein this assertion is generalised for the so-called $p$-variation. 
In \citeN{Woerner07} a \emph{normed} $p$-variation is introduced and for the time-changed processes a relation to the Blumenthal-Getoor index is derived.
\par
In particular, for L\'evy processes $L$ it is shown under some assumptions on the characteristic triplet, that the normed $p$-variation for $0<p<\beta$ with $p\neq\beta-1$, exists on a finite time interval $[0,T]$,
\begin{align}\label{normed-pvar}
\Delta_n^{1 - p/\beta} \sum_{i=1}^{n-1} \big|L_{(i+1)\Delta_n}  - L_{i\Delta_n}\big|^{p} \xrightarrow{P}{} V_p(L)\,,
\end{align}
where $V_p(L)$ is a finite number, $\beta$ is the Blumenthal-Getoor index, and for $n\to \infty$ the partition of the time interval $[0,T]$ refines uniformly, $\Delta_n\downarrow 0$, see \cite[Theorem 1]{Woerner07} and \cite[Corollary 1]{Woerner03}. The additional assumptions made therein, concern the L\'evy-measure, that is assumed to have a density with a certain Taylor expansion around the origin, and a special choice of the drift.
\par
In order to compare the Sobolev index with the Blumenthal Getoor index, we introduce another index $\gamma$ that, similar to the Blumenthal-Getoor index, quantifies the intensity of small jumps of the process.
\begin{lemma}\label{levy-index}
Let $L$ be a L\'evy process with characteristics $(b,c,F)$ and with Blumenthal-Getoor index $\beta$. We define the index
\begin{eqnarray*}
 \gamma := \sup\left\{ \alpha>0 \,\Big|\, \liminf_{r\downarrow 0} r^{\alpha-2} \int_{[-r,r]} |x|^2 F(\dd x) >0 \right\}\,.
\end{eqnarray*}
We have
$$
\beta\ge \gamma\,.
$$
\end{lemma}
\begin{proof}
For $0<\alpha<2$ and $0<r<1$ we have $ r^{\alpha-2} \int_{-r}^r |x|^2 F(\dd x) \le  \int_{-r}^r |x|^{\alpha} F(\dd x) $, since
\begin{eqnarray*}
\int_{-r}^r |x|^2 r^{\alpha-2} F(\dd x) 
\,\le\,
\int_{-r}^r  |x|^{\alpha} F(\dd x) \,.
\end{eqnarray*}
If
$$
 \liminf_{r\downarrow 0} r^{\alpha-2} \int_{-r}^r |x|^2 F(\dd x) >0\,,
$$
then there exists a constant $C>0$ with $\int_{-r}^r |x|^{\alpha} F(\dd x) > C$ for all $0<r$ smaller than some $\epsilon>0$. Hence $\int_{-1}^1|x|^{\alpha} F(\dd x) =\infty$ follows from $F(\{0\}) = 0$. This means that for every $\alpha<\gamma$ we have $\alpha\le \beta$ whence $\gamma \le \beta$.
\notiz{
Note that $\gamma = \inf\left\{ \alpha>0 \,\Big|\, \liminf_{r\downarrow 0} r^{\alpha-2} \int_{[-r,r]} x^2 F(\dd x)  = 0 \right\}$. 
\par
Conversely $\int_{-1}^1|x|^{\alpha} F(\dd x) <\infty$ entails 
$$
\liminf_{r\downarrow 0} \int_{-r}^r |x|^{\alpha} F(\dd x) =0
$$
whence $\liminf_{r\downarrow 0} r^{\alpha-2} \int_{[-r,r]} x^2 F(\dd x)  = 0$.
}
\end{proof}
\noindent The index $\gamma$ quantifies the intensity of small jumps of the L\'evy process, hence it is also a measure for the regularity of the underlying distribution. \citeN[Proposition 28.3]{Sato} shows the following remark.
\begin{remark}
If we have $0<\gamma<2$, then the distribution $\mu_1=P^{L_1}$ does possess a smooth Lebesgue density.
\end{remark}
\noindent The rest of the section is dedicated to the relation between the index $\gamma$, the Blumenthal Getoor-index, and the Sobolev index of a real-valued L\'evy process. We therefore restrict ourselves to L\'evy processes that take values in $\rr$.
\begin{proposition}\label{lem-blumenthal-getoor}
Let  $L$ be a real-valued L\'evy process with characteristic triplet $(b,0,F)$.
If the index $\gamma$ satisfies $\gamma\in(0,2)$, then the symbol $A$ of $L$ satisfies a G{\aa}rding-condition for any index $\alpha<\gamma$.
\end{proposition}
\begin{proof}
Let us write down the real part of the symbol,
\begin{eqnarray*}
 \Re \big( A(\xi) \big)
=
 \int \big(1- \cos(\xi y) \big) F(\dd y)\,.
\end{eqnarray*}
As in the proof of \citeN[Proposition 28.3]{Sato} we further conclude
\begin{eqnarray}\label{sato_hilft}
\lefteqn{\int(1-\cos(vy))F(\dd y)}\qquad\nonumber\\
&=&
2\int_{|v||y|\le\pi}\sin^2\Big(\frac{vy}{2}\Big)F(\dd y)+2\int_{|v||y|>\pi}\sin^2\Big(\frac{vy}{2}\Big)F(\dd y)\nonumber\\
&\ge&
2\int_{|v||y|\le\pi}\frac{2}{\pi^2}v^2y^2F(\dd y)+2\int_{|v||y|>\pi}\sin^2\Big(\frac{vy}{2}\Big)F(\dd y)\nonumber\\
&\ge&
c'\int_{|v||y|\le\pi}v^2y^2F(\dd y)
\end{eqnarray}
for a positive constant $c'$. On the other hand, for $\alpha < \gamma$ we have
\begin{eqnarray*}
\liminf_{r\downarrow 0} r^{\alpha-2} \int_{[-r,r]} x^2 F(\dd x) >0 \,,
\end{eqnarray*}
hence there exists a constant $c_1>0$ and $\epsilon >0$, such that
\begin{eqnarray*}
\int_{[-r,r]} x^2 F(\dd x) \ge c_1 r^{2-\alpha} \qquad\text{for every } r<\epsilon \,.
\end{eqnarray*}
That is, for every $\alpha<\gamma$ there exists an  $N>0$ and a constant $c_c>0$ with
$$
\int_{|v||y|\le\pi}v^2y^2F(\dd y)\ge c_c|v|^\alpha\qquad\hbox{for all }v\,\hbox{with }|v|>N\,.
$$
Altogether, we have
\begin{eqnarray*}
 \Re \big( A(\xi) \big)
\ge
 c|\xi|^\alpha \1_{|\xi| >N}
\ge
 c|\xi|^\alpha - c_2
\end{eqnarray*}
with $c>0$ and $c_2 = c|N|^\alpha$.
\end{proof}
\begin{proposition}\label{satz-momenteundblumenthal}
Let $L$ be a real-valued L\'evy process with characteristics $(b,0,F)$, and s.t. its symbol satisfies the G{\aa}rding-condition with $0<Y<2$. Then
$$
\int_{-1}^1 |x|^\alpha F(\dd x) = \infty\quad\text{for all }\alpha< Y<2\,.
$$
In particular the Blumenthal-Getoor index $\beta$ of the process is bigger or equal to $Y$ ($\beta \ge Y$).
\end{proposition}
\begin{proof}
From the assumption we know that there exist constants $C_1>0$ and $C_2\ge 0$ and indexes $0<Y'<Y<2$ with
$$
\int \big(1-\cos( u x) \big) F(\dd x) \ge C_1|u|^Y - C_2\left(1+|u|^{Y'}\right)\,.
$$
Thus for every $\epsilon>0$, the inequality
$$
\int_{-\epsilon}^\epsilon \big(1-\cos (u x)) \big) F(\dd x) \ge C_1|u|^Y - C_2|u|^{Y'} - C_\epsilon
$$
holds for $C_\epsilon = C_2 + 2F\big((-\epsilon,\epsilon)^c\big)$. Since for every $0<\alpha< 2$ there exists a constant $C(\alpha)>0$ with $1-\cos(y) \le C(\alpha) |y|^\alpha$ for all $y\in \rr$, we are able to conclude for any fixed $\epsilon>0$ that
$$
C_1|u|^Y - C_2|u|^{Y'} - C_\epsilon \le \int_{-\epsilon}^\epsilon \big(1-\cos( u x ) \big) F(\dd x)  \le C(\alpha)  \int_{-\epsilon}^\epsilon | u x |^\alpha F(\dd x)
$$
for all $u\in \rr$, resp.
$$
\frac{C_1}{C(\alpha)}|u|^{Y-\alpha} - \frac{C_2}{C(\alpha)}|u|^{Y'-\alpha} - \frac{C_\epsilon}{C(\alpha)} |u|^{-\alpha}\le \int_{-\epsilon}^\epsilon |x|^\alpha F(\dd x)\quad\text{for all }u\in \rr\setminus\{0\} \,.
$$
For every $\epsilon>0$ the left hand side of the inequality diverges for $|u|\to \infty$, if $\alpha< Y$. Thus we can conclude $\infty = \int_{-\epsilon}^{\epsilon} |x|^\alpha F(\dd x)$ for every $\epsilon>0$.
\end{proof}
\noindent Proposition \ref{satz-momenteundblumenthal} and Proposition \ref{Gard-folgtDichte} together yield the following result.
\begin{remark}\label{rem-rel-sob-blum}
The relation between both indexes, the Blumenthal-Getoor and the Sobolev index, bridges the path properties and the distribution of the process.
If the L\'evy process has a Sobolev index, its distribution is smooth. Furthermore, its paths are of unbounded variation if the Sobolev index is bigger or equal to $1$.
\end{remark}

This relation can be studied more extensively using a certain type of Feynman-Kac formula and results on $p$-variations of the process. \cite{Woerner07} shows convergence in probability of the normed $p$-variation \eqref{normed-pvar} under appropriate conditions on the L\'evy process. This can be interpreted as a result on the intensity of oscillations of the paths of the process. On the other hand, Feynman-Kac formulas allow us to interpret the degree of smoothness of the solution of the PIDE as an effect that directly stems from the distribution. An appropriate Feynman-Kac formula that allows us to distinguish between different degrees of smoothing is given in \cite[Theorem IV.9]{PhdGlau}.

To conclude, let us point out that the results in this article have an obvious extension to the case of time-inhomogeneous L\'evy process when one requires continuity and G{\aa}rding condition \emph{uniformly in time}. Moreover, for applications to option pricing, continuity and G{\aa}rding condition are studied for an analytical extension of the symbol to a certain domain in the complex plane in the article \cite{EberleinGlau2012}. For a more extensive study of multivariate processes, anisotropic Sobolev-Slobodeckii spaces are the appropriate spaces. Moreover, an extension of the framework to affine processes is a current research topic. This extension is not obvious, since the symbol of an affine process is affine in the state space, hence uniform bounds with respect to the Sobolev-Slobodeckii norms are not available.
 

\appendix
\section{}\label{appendixA}
The following lemma relates the behaviour of the symbol $A(u)$ for $|u|\to\infty$ with the behaviour of the L\'evy measure $F$ around the origin.
\par
Again, we use Landau's symbol $O$ to indicate the asymptotic behaviour; here we look at the behaviour of a function around the origin. 
More precisely we write $f(x) = O\big(g(x)\big)$ for $x\to 0$ if there exist positive constants $M$ and $N$ such that $\frac{|f(x)|}{|g(x)|} \le M$ for all $|x|<1/N$.

As generally assumed in Section \ref{sec-abscontF}, let $L$ be a real-valued L\'evy process that is a special semimartingale with characteristic triplet $(b,0,F)$ w.r.t. $h(x) = x$. Furthermore assume $F(\dd x) = f(x) \dd x$ for the L\'evy measure $F$ and we denote by $f_s$ the symmetric and by $f_{as}$ the antisymmetric part of the density function $f$.

\begin{lemma}\label{lem_sob_ordn} Let $0<Y<2$.
 \begin{itemize}
  \item [a)]
If $f_s(x) = O\left( \frac{1}{|x|^{1+Y}}\right)$ for $x\to 0$, then there exists a constant $C\ge 0$ with
$$
0 \le \Re\left( A^{f}(u)\right) = A^{f_s}(u) \le C\left(1+|u|^{Y}\right)\qquad \text{ for all } u \in \rr\,.
$$
\item [b)]
If $f_s(x) = \frac{C}{|x|^{1+Y}} + g(x)$ with $g(x)=O\left( \frac{1}{|x|^{1+Y-\delta}}\right)$ for $x\to 0$ with some $0<\delta$ and  $C>0$, then there exist constants $C_1>0$, $C_2\ge 0$ and $Y'\in(0,Y)$ such that
$$
\Re\left(A^{f}(u)\right) = A^{f_s}(u) \ge C_1|u|^Y - C_2\big(1+|u|^{Y'}\big) \qquad \text{for all }u \in \rr\,.
$$
\item[c)]
If $f_{as}(x) = O\left( \frac{1}{|x|^{1+Y}}\right)$ for $x\to 0$ with $0< Y$ and $Y \neq 1$,
then there exist constants $C,\,C_1\ge 0$ with
$$ 
\left|\Im\left(A^{f}(u)\right)\right| 
=
 \left|A^{f_{as}}(u)\right| 
\le
 C \left(1+ |u| + |u|^Y\right) 
\le
C_1\left(1+ |u|^{\max[1,Y]}\right)
$$
for every $u\in\rr$.
\item[d)]
Let $f_{as}(x) = O\left(\frac{1}{|x|^{1+Y}}\right)$ for $x\to 0$ with $Y\in(0,1)$ and assume $\int|x| f(x) \dd x <\infty$ i.e. the paths of the process are a.s. of finite variation.

If $L$ is a L\'evy process with characteristic triplet $(\int xF(\dd x),0,F)$ w.r.t. the truncation function $h(x) = x$,
then there exists a constant $C\ge 0$ with
$$
\Big|\Im\big(A(u)\big)\Big| \le C\left(1+|u|^Y\right)\qquad\text{for all }u\in\rr\,.
$$
\end{itemize}
\end{lemma}

\begin{proof}
Proof of a): For every $\epsilon>0$ and arbitrary  $u\in\rr$ we have
$$
A^{f_s}(u) = \int_{-\epsilon}^\epsilon \left( 1- \cos(ux)\right) f_s(x) \dd x + \int_{(-\epsilon,\epsilon)^c}\left( 1- \cos(ux)\right) f_s(x) \dd x$$
with
\begin{align*}
0 &\le \int_{(-\epsilon,\epsilon)^c}\left( 1- \cos(ux)\right) f_s(x) \dd x 
\le 2 \int_{(-\epsilon,\epsilon)^c}f_s(x) \dd x 
=: C(\epsilon)\,.
\end{align*}
If we choose $\epsilon>0$ small enough, we get
\begin{eqnarray*}
\lefteqn{ \int_{-\epsilon}^\epsilon \left( 1- \cos(ux)\right) f_s(x) \dd x}\qquad\\
&\le&
C_1(\epsilon) \int_{-\epsilon}^\epsilon \left( 1- \cos(ux)\right) \frac{1}{|x|^{1+Y}}\dd x \\
&=&
 C_1(\epsilon) |u|^Y \int_{-\epsilon |u|}^{\epsilon |u|}  \frac{ 1- \cos x}{|x|^{1+Y}} \dd x \\
&=&
2 C_1(\epsilon) |u|^Y \left( \int_0^1 \frac{1-\cos x}{|x|^{1+Y}} \dd x + \int_{1}^{\epsilon |u|} \frac{1-\cos x}{|x|^{1+Y}} \dd x \right)
\end{eqnarray*}
with a constant $C_1(\epsilon)>0$ only depending on $\epsilon$. Furthermore,
$$
\int_0^1 \frac{1-\cos x}{|x|^{1+Y}} \dd x 
\le \frac{1}{2} \int_0^1 \frac{x^2}{|x|^{1+Y}} \dd x 
= \frac{1}{2} \int_0^1 x^{1-Y} \dd x = \frac{1}{2(2-Y)} < \infty\,,
$$
since $Y<2$. The second integral is negative for $\epsilon|u|<1$, and for $1<\epsilon|u|$ we get
$$
 0\le \int_{1}^{\epsilon |u|} \frac{1-\cos x}{|x|^{1+Y}} \dd x 
\le
\int_1^{\epsilon |u|} \frac{2}{|x|^{1+Y}} \dd x 
=
\frac{2}{Y} \left( - \big(\epsilon|u|\big)^{-Y} + 1 \right) 
\le \frac{2}{Y}\,.
$$
So there exist $\epsilon>0$ and $C_2(\epsilon)\ge 0$ with
\begin{equation}\label{inteps}
\int_{-\epsilon}^\epsilon \left( 1 - \cos(ux) \right) f_s(x) \dd x \le 
C_2(\epsilon) |u|^Y\,.
\end{equation}
For an appropriate choice of $\epsilon$ we directly obtain the assertion of a).
\par
Proof of b): The first equality of the assertion is given by \eqref{eq-reA}. For every $\epsilon>0$ we have
\begin{eqnarray*}
 A^{f_s}(u) 
&=&
 \int_{-\epsilon}^\epsilon \left( 1- \cos(ux)\right) f_s(x) \dd x +\int_{(-\epsilon,\epsilon)^c}\left( 1- \cos(ux)\right) f_s(x) \dd x \,.
\end{eqnarray*}
Since $\left( 1- \cos(ux)\right)\ge 0$ this yields $A^{f_s}(u) \ge\int_{-\epsilon}^\epsilon \left( 1- \cos(ux)\right) f_s(x) \dd x$. By inserting the assumption on $f_s$, we obtain for $\epsilon$ small enough
\begin{eqnarray*}
 A^{f_s}(u) 
&\ge &
\int_{-\epsilon}^\epsilon \left( 1- \cos(ux)\right) \frac{C}{|x|^{1+Y}} \dd x + \int_{-\epsilon}^\epsilon \left( 1- \cos(ux)\right) g(x)  \dd x \\
&\ge&
C \int_{-\epsilon}^\epsilon \frac{1- \cos(ux)}{|x|^{1+Y}} \dd x - \left|\int_{-\epsilon}^\epsilon \left( 1- \cos(ux)\right) g(x)  \dd x \right| \,.
\end{eqnarray*}
For the first integral, a computation similar to \eqref{sato_hilft} in the proof of Proposition \ref{lem-blumenthal-getoor} yields
\begin{align*}
\int_{-\epsilon}^\epsilon \frac{1- \cos(ux)}{|x|^{1+Y}} \dd x 
&\ge
c'\int_{|ux|\le \pi} \frac{x^2 u^2}{|x|^{1+Y}} \dd x - C_1(\epsilon)\\
&=
c'|u|^Y \int_{|x|\le\pi} |x|^{1-Y} \dd x - C_1(\epsilon)
= C_2|u|^Y - C_1(\epsilon)
\end{align*}
with the positive constants $C_1$ and $C_2(\epsilon)$ given by $C_1(\epsilon)= 2 \int_{(-\epsilon,\epsilon)^c}\frac{1}{|x|^{1+Y}}\dd x$ and $C_2:=c'\int_{|x|\le\pi} |x|^{1-Y} \dd x$.
In order to find an upper bound for the second integral, let us assume $\epsilon<1$.
Then $|x|^{-1-Y+\delta}\le|x|^{-1-Y+\delta'}$ with $\delta':=\min\{Y/2,\delta\}$ for $|x|<\epsilon$.
Arguing along the same lines as in the proof of equation \eqref{inteps} yields since $\delta'<Y$
\begin{align*}
0\le\int_{-\epsilon}^\epsilon \frac{1-\cos(ux)}{|x|^{1+Y-\delta}} \dd x \le
\int_{-\epsilon}^\epsilon \frac{1-\cos(ux)}{|x|^{1+Y-\delta'}} \dd x\le C(\epsilon)|u|^{Y-\delta'}
\end{align*}
for some constant $C(\epsilon)>0$.
Fixing some appropriate $\epsilon>0$ we have
$$
A^{f_s}(u)  \ge C_1|u|^Y - C_2\Big( 1+ |u|^{Y-\delta'} \Big)
$$
for a strictly positive constant $C_1,\,C_2\ge0$ and $0<\delta'<Y$.
\par

Proof of c): The first equality of the assertion is given by equation \eqref{eq-reA-imA}.
For every $u\in\rr$ we have $\left|A^{f_{as}}(u)\right| \le \int \left|ux - \sin(ux)\right| \left|f_{as}(x)\right| \dd x$ and if we choose $\epsilon>0$ small enough, Lemma \ref{bem_fas-kl-fs} allows us to conclude
\begin{eqnarray*}
\lefteqn{ \int_{(-\epsilon,\epsilon)^c}\!  \left|ux - \sin(ux)\right| \left|f_{as}(x)\right| \dd x}\qquad\\
&\le&
|u| \int _{(-\epsilon,\epsilon)^c}\! |x| \left|f_{as}(x)\right| \dd x +  \int _{(-\epsilon,\epsilon)^c}\! \left|f_{as}(x)\right| \dd x \\
&\le&
 |u| \int _{(-\epsilon,\epsilon)^c} |x| \left|f_{s}(x)\right| \dd x  +  \int _{(-\epsilon,\epsilon)^c} \left|f_{s}(x)\right| \dd x \\
&=&
 |u| \int _{(-\epsilon,\epsilon)^c} |x| F(\dd x) +   \int _{(-\epsilon,\epsilon)^c}F(\dd x) \\
&=:&
C_1(\epsilon)|u| + C_2(\epsilon)
\end{eqnarray*}
with nonnegative constants $C_1(\epsilon)$ and $C_2(\epsilon)$. From the assumption on $f_{as}$ we get
\begin{eqnarray*}
\lefteqn{\int_{-\epsilon}^{\epsilon} \left| ux - \sin(ux)\right|\left|f_{as}(x)\right| \dd x}\qquad\\
&\le&
C(\epsilon) \int_{-\epsilon}^{\epsilon} \left| ux - \sin(ux)\right| \frac{1}{|x|^{1+Y}} \dd x \\
&=&
C(\epsilon) |u|^Y  \int_{-\epsilon |u|}^{\epsilon |u|} \frac{|x-\sin(x)|}{|x|^{1+Y}} \dd x \\
&=&
2C(\epsilon) |u|^Y  \Bigg( \int_0^1 \frac{x-\sin(x)}{|x|^{1+Y}} \dd x + \int_1^{\epsilon |u|}  \frac{x-\sin(x)}{|x|^{1+Y}} \dd x    \Bigg)\,,
\end{eqnarray*}
where the first integral is finite since $Y<2$.
As before, the second integral is negative for $\epsilon|u|<1$ and for $1<\epsilon|u|$ we have
\begin{eqnarray*}
 \int_1^{\epsilon |u|}  \frac{x}{|x|^{1+Y}} \dd x 
&=&
\int_1^{\epsilon |u|} x^{-Y} \dd x 
=
\frac{\epsilon^{1-Y}}{1-Y} |u|^{1-Y} - \frac{1}{1-Y}
\end{eqnarray*}
since $Y\neq 1$ and
\begin{eqnarray*}
 - \int_1^{\epsilon|u|} \frac{\sin(x)}{|x|^{1+Y}} \dd x
\le
\int_1^{\epsilon|u|} x^{-1-Y} \dd x
=
-\frac{|u|^{-Y}}{ Y \epsilon^Y} + \frac{1}{Y} 
\le C_3(\epsilon)\big(1+|u|^{-Y} \big)
\end{eqnarray*}
with some constant $C_3(\epsilon)>0$.
Combining these estimates and fixing some $\epsilon>0$, we obtain the assertion of part c).
\par
Proof of d): Since $f_{as}$ is antisymmetric and $\int |x f_{as}(x)| \dd x \le \int|x| f(x) \dd x <\infty$ by Lemma \ref{bem_fas-kl-fs}, we obtain
$$
A^{f_{as}}(u)
=
i \Im\left(A^{f}(u)\right)
=
i \int \sin(ux) f_{as}(x) \dd x - iu \int xf(x)\dd x \,.
$$
Furthermore since the drift is given by $\int x f(x)\dd x$ we have
\begin{equation*}
 \Big|\Im\big(A(u)\!\big)\!\Big|\!
=
\!\left| \int \!\sin(ux) f_{as}(x)\! \dd x\right| 
\le
\! \int_{(-\epsilon,\epsilon)^c} \!\!\big|f_{as}(x)\big| \dd x  + \!\int_{-\epsilon}^\epsilon \!\!\big| \sin(ux)\big| \big|f_{as}(x)\!\big| \dd x
\end{equation*}
hence by the assumption on $f_{as}$ we obtain
\begin{eqnarray*}
\Big|\Im\big(A(u)\big)\Big|
&\le&
F\big( (-\epsilon,\epsilon)^c \big) + C(\epsilon) \int _{-\epsilon}^\epsilon  \frac{| \sin(ux) |}{|x|^{1+Y}} \dd x \\
&\le&
C_1(\epsilon) + C(\epsilon) |u|^Y \int_{-\infty}^\infty \frac{|\sin(x)|}{|x|^{1+Y}} \dd x\\
&=&
C_1(\epsilon)+ C_2(\epsilon)|u|^Y
\end{eqnarray*}
with positive constants $C(\epsilon)$, $C_1(\epsilon)$ and $C_2(\epsilon)$. Choosing $\epsilon>0$ yields the result.
\end{proof}